\documentclass[11pt]{amsart}
\usepackage{mathrsfs}
\usepackage{amsfonts}
\usepackage{amssymb}
\usepackage{amsxtra}
\usepackage{dsfont}
\usepackage{color,xcolor}
\usepackage{graphicx}
\usepackage{enumitem}
\usepackage[english, german]{babel}

\usepackage[T1]{fontenc}
\usepackage{newtxmath}
\usepackage{newtxtext}
\usepackage{csquotes}
\usepackage[compress, sort]{cite}

\usepackage[margin=2.5cm, centering]{geometry}
\usepackage[colorlinks,citecolor=blue,urlcolor=blue,bookmarks=true]{hyperref}
\hypersetup{
	pdfpagemode=UseNone,
	pdfstartview=FitH,
	pdfdisplaydoctitle=true,
	pdfborder={0 0 0}, % No link borders
	pdftitle={Mass Functions and Asymptotic Behavior of Caloric Functions on Affine Buildings},
	pdfauthor={Effie Papageorgiou and Bartosz Trojan},
	pdflang=en-US
}

\newcommand{\pl}[1]{\foreignlanguage{polish}{#1}}
\newcommand{\de}[1]{\foreignlanguage{german}{#1}}

\newcommand{\CC}{\mathbb{C}}

\newcommand{\NN}{\mathbb{N}}

\newcommand{\RR}{\mathbb{R}}

\newcommand{\ZZ}{\mathbb{Z}}

\newcommand{\calC}{\mathcal{C}}
\newcommand{\calL}{\mathcal{L}}

\newcommand{\calO}{\mathcal{O}}
\newcommand{\calM}{\mathcal{M}}
\newcommand{\calB}{\mathcal{B}}

\newcommand{\calA}{\mathcal{A}}
\newcommand{\calN}{\mathcal{N}}

\newcommand{\calH}{\mathcal{H}}

\newcommand{\scrX}{\mathscr{X}}
\newcommand{\scrC}{\mathscr{C}}
\newcommand{\scrA}{\mathscr{A}}
\newcommand{\scrB}{\mathscr{B}}
\newcommand{\scrD}{\mathscr{D}}
\newcommand{\scrS}{\mathscr{S}}
\newcommand{\scrN}{\mathscr{N}}
\newcommand{\frakA}{\mathfrak{a}}

\newcommand{\bfc}{\mathbf{c}}

\newcommand{\St}{\operatorname{St}}

\newcommand{\Aut}{\operatorname{Aut}}
\newcommand{\GL}{\operatorname{GL}}
\newcommand{\Aff}{\operatorname{Aff}}
\newcommand{\height}{\operatorname{ht}}

\newcommand{\vphi}{\varphi}

\newcommand{\norm}[1]{\lvert {#1} \rvert}
\newcommand{\abs}[1]{\lvert {#1} \rvert}
\newcommand{\sprod}[2]{\langle {#1}, {#2} \rangle}

\newcommand{\dist}{\operatorname{dist}}

\newcommand{\inter}{\operatorname{int}}

\newcommand{\cspan}{\operatorname{\mathbb{C}-span}}
\newcommand{\zspan}{\operatorname{\mathbb{Z}-span}}

\newcommand{\supp}{\operatorname{supp}}

\theoremstyle{plain}
\newcounter{thm}

\newtheorem{main_theorem}[thm]{Theorem}

\newtheorem{theorem}{Theorem}[section]
\newtheorem{proposition}[theorem]{Proposition}
\newtheorem{lemma}[theorem]{Lemma}
\newtheorem{corollary}[theorem]{Corollary}

\newtheorem*{theorem*}{Theorem}

\theoremstyle{definition}

\newtheorem{remark}{Remark}

\numberwithin{equation}{section}

\title[Mass Functions and Caloric Functions]
{Mass Functions and Asymptotic Behavior of Caloric Functions on Affine Buildings}

\author{Effie Papageorgiou}
\address{
	\de{
		Effie Papageorgiou \\
		Institut f{\"u}r Mathematik \\
		Universit\"at Paderborn\\
		Warburger Str. 100\\ 
		D-33098 Paderborn\\
		Germany}
}
\email{papageoeffie@gmail.com}

\author{Bartosz Trojan}
\address{
	\pl{
		Bartosz Trojan\\
		Wydzia\l{} Matematyki,
		Politechnika Wroc\l{}awska\\
		Wyb. Wyspia\'{n}skiego 27\\
		50-370 Wroc\l{}aw\\
		Poland}
}
\email{bartosz.trojan@gmail.com}

\keywords{affine building, random walk, caloric functions, asymptotic formula, excotic building}

\subjclass[2020]{Primary: 35K08, 35B40, 51E24, 58J35, 60B15.
				Secondary: 20E42, 20F55, 22E35, 43A90, 60G50.}

\begin{document}

\selectlanguage{english}

\begin{abstract}
	We study the large-time asymptotic behavior of solutions to the discrete-time heat equation, i.e., caloric functions, 
	on affine buildings, including those without transitive group actions. For each $p \in [1, \infty]$, we introduce a notion 
	of a $p$-mass function and prove that caloric functions with initial data belonging to certain weighted-$\ell^1$ spaces or
	to the radial $\ell^1$ class, asymptotically decouple as the product of this mass function and the heat kernel. 
	These results extend classical analogues from Euclidean spaces and symmetric spaces of non-compact type to the
	non-Archimedean setting, and remain valid even for exotic buildings beyond the Bruhat--Tits framework.
	We characterize the spatial concentration of heat kernels in $p$-norms and describe the geometry of associated critical
	regions. Our results highlight substantial differences in the asymptotic regimes depending on the value of $p$, and clarify 
	the interplay between volume growth and heat diffusion. 
\end{abstract}

\maketitle

\section*{Introduction}
At the heart of harmonic analysis lies the study of solutions to the heat equation with the central r\^ole 
played by its fundamental solution $k$, called heat kernel. In Euclidean space $\RR^d$, the heat kernel
is given by
\[
	k_t(x) = (4 \pi t)^{-\frac{d}{2}} e^{-\frac{|x|^2}{4t}},
\]
representing the distribution of the heat at $x \in \RR^d$ at time $t > 0$ emanating at the point source at
the origin. Its smoothness, decay, and scaling properties make it a powerful tool in harmonic analysis, PDE theory, and 
probability theory. For $p \in [1, \infty]$, a straightforward computation yields
\[
	\|k_t\|_{L^p} = 
	(4 \pi t)^{-\frac{d}{2p'}} p^{-\frac{d}{2p}}
\]
where $p'$ is the H\"older conjugate exponent to $p$, defined by $1/p + 1/p' = 1$. If $u$ is a caloric function
with the initial datum $f$ belonging to $L^1(\RR^d)$, that is it solves
\begin{align*}
	\begin{cases}
		\partial_{t}u(t,x) = \Delta_{\mathbb{R}^{d}} u(t,x), &\text{for } t>0 \text{ and } x \in \RR^d\\[0,5pt]
		u(0,x) =f(x),
	\end{cases}
\end{align*}
then $u(t, x) = k_t * f(x)$. Hence, by the Young's convolution inequality we have
\[
	\|u(t; \cdot)\|_{L^p} \leqslant \|f\|_{L^1} \|k_t\|_{L^p}.
\]
In particular, $u(t; \cdot)$ uniformly goes to zero as time goes to infinity. The classical result says that $u(t, \cdot)$
for large time looks like $M_1(f) k_t$, where
\[
	M_1(f) = \int_{\RR^d} f(x) {\: \rm d} x.
\]
More precisely, for each $p \in [1, \infty]$,
\[
	\lim_{t \to \infty} \frac{1}{\|k_t\|_{L^p}} \big\| u(t; \cdot) - M_1(f) k_t \big\|_{L^p} = 0.
\]
This result has been also studied on certain manifolds of non-negative Ricci curvature under the assumption that 
the volume function satisfies the doubling condition, see \cite{GPZ2023}. In the forthcoming paper \cite{GPT2025},
we investigate the discrete counterpart. On the other end, V\'azquez in \cite{VazHyp},
showed that for real hyperbolic spaces $\mathbb{H}_n$, which are examples of negatively curved manifolds,
if the initial datum $f \in L^1(\mathbb{H}_n)$ is \emph{radial} and
\[
	M_1(f) = \int_{\mathbb{H}_n} f(x) {\: \rm d} x,
\]
then
\begin{equation}
	\label{eq:84}
	\lim_{t \to +\infty} 
	\frac{1}{\|k_t\|_{L^1(\mathbb{H}_n)}}
	\|u(t; \cdot) - M_1(f) k_t \|_{L^1(\mathbb{H}_n)} = 0.
\end{equation}
However, there is a radial $f$ for which $\|u(t; \cdot) - M_1(f) k_t \|_{L^\infty(\mathbb{H}_n)}$ is not 
$o(\|k_t\|_{L^\infty(\mathbb{H}_n})$. It has also been shown that for initial datum which is not radial the convergence in
\eqref{eq:84} may fail. The work \cite{VazHyp} has been extended to higher rank symmetric spaces of non-compact type in 
\cite{APZ2023}. V\'azquez's negative result for $L^\infty(\mathbb{H}_n)$
motivated the work \cite{NRS24}. To be more precise, let $\mathbb{X} = G/K$ be a symmetric space of non-compact type. 
In \cite{NRS24}, for each $p \in [1, \infty]$, the authors introduced certain function spaces $\calL^p(\mathbb{X})$ 
containing $L^1(\mathbb{X})$, for which the mass function $M_p(f)$ is defined by
\[
	M_p(f) = 
	\begin{cases}
		\hat{f}\big(i\rho(2/p-1)\big) & \text{if } p \in [1, 2), \\
		\hat{f}(0) &\text{if } p \in [2, \infty]
	\end{cases}
\]
where $\hat{f}$ denotes the spherical transform of $f$. The authors then showed that if $f$ is a bi-$K$-invariant 
initial datum belonging to $\calL^p(\mathbb{X})$, then
\begin{equation}
	\label{eq:85}
	\lim_{t \to +\infty} \frac{1}{\|k_t\|_{L^p(\mathbb{X})}} \big\|u(t; \cdot) - M_p(f) k_t\big\|_{L^p(\mathbb{X})} = 0.
\end{equation}
However, the convergence in \eqref{eq:85} may fail if $f$ is not bi-$K$-invariant. To study initial data that are not
necessarily bi-$K$-invariant, in \cite{P1} the first-named author to any compactly supported function $f$ on $\mathbb{X}$ 
associated a function $M_p(f)$ on $\mathbb{X}$ for which \eqref{eq:85} again holds true.
	
The aim of this paper is to study the asymptotic behavior of caloric functions on certain discrete spaces. Our main goal
is to investigate the non-Archimedean counterparts to symmetric spaces of non-compact type, that is, affine buildings.
They are complementary pieces at the heart of Lie theory  \cite{Helgason1979, BruhatTits1972}. 
Apart from strong topological differences, real and non-Archimedean simple Lie groups share many combinatorial and geometric 
properties. From the geometric point of view, they both act, with strong transitivity properties, on contractible spaces 
carrying nice non-positively curved complete distances. The corresponding spaces for groups over totally disconnected local
fields are Bruhat--Tits buildings. Bruhat--Tits buildings, or more generally affine buildings when no group is assumed to 
act transitively on them, are unions of Euclidean tilings, called apartments, playing the r\^ole of maximal flats in 
symmetric spaces. Apartments contain Weyl cones, also called sectors.

To precisely state our first result, we need to introduce some notation. Let $\Phi$ be the type of the building, that is 
$\Phi$ is the affine root system in $\mathfrak{a}$, where $\mathfrak{a}$ is the $r$-dimensional Euclidean space on which 
apartments are modeled, and let $\Phi^{++}$ denote the set of indivisible positive roots in $\Phi.$ By $W$ we denote 
the corresponding (spherical) Weyl group. Given a transition function $k$ of the isotropic finite range random walk on 
good vertices $V_g$ of the building (called admissible for short), we set
\[
	\begin{aligned}
	k(1; x, y) &= k(x, y), \\
	k(n+1; x, y) &= \sum_{z \in V_g} k(n; x, z) k(z, y) \qquad (n, x, y) \in \NN_0 \times V_g \times V_g.
	\end{aligned}
\]
We consider the corresponding averaging operator acting on functions on $V_g$ as
\[
	A f(x)=\sum_{y \in V_g} k(y, x) f(y).
\]
The Gelfand--Fourier transform of $A$, denoted by $\hat{A}$, is a $W$-invariant exponential polynomial expressed 
as a combination of Macdonald's spherical functions $P_\omega$. Let $\varrho$ be the spectral radius of the random walk 
and $\kappa=\varrho^{-1}\widehat{A}$. We fix a good vertex $o$. Let $\eta$ be a properly weighted sum of positive roots,
see \eqref{eq:6} for precise definition. At this stage it is enough to remember that $\eta$ belongs to the positive Weyl
chamber. In the following theorem we summarize the main results proved in Theorems \ref{thm:1}, \ref{thm:2} and \ref{thm:3}.
\begin{main_theorem}
	\label{thm:A}
	Let $k$ be a transition kernel of an admissible random walk on good vertices of an affine building 
	$\mathscr{X}$. Then
	\begin{enumerate}[label=(\roman*), ref=\roman*]
		\item if $p \in [1, 2)$,
		\[
			\|k_n(o, \cdot)\|_{\ell^p}\approx n^{-\frac{r}{2p'}} 
			\varrho^{n} \kappa\left(\eta\left(\tfrac{2}{p}-1\right)\right)^{n};
		\]
		\item if $p=2$,
		\[
			\|k_n(o, \cdot)\|_{\ell^2}\approx n^{-\frac{r}{4}-\frac{|\Phi^{++}|}{2}} \varrho^{n};
		\]
		\item if $p \in (2, \infty]$,
		\[
			\|k_n(o, \cdot)\|_{\ell^p} \approx n^{-\frac{r}{2}-|\Phi^{++}|} \varrho^{n}.
		\]
	\end{enumerate}
\end{main_theorem} 
Let us observe that for $p \in (2, \infty]$, the asymptotic formula for the norm $\|k_n(o, \cdot)\|_{\ell^p}$
is independent of $p$. For $p \in [1, 2)$ one only sees the rank and not the pseudodimension $|\Phi^{++}|$, and the exponential 
decay in $n$ is slower than for $p \in [2, \infty)$. And lastly, the case $p = 2$ is not a limit of the cases $[1, 2)$ nor 
$(2, \infty]$.

In each of Theorems \ref{thm:1}, \ref{thm:2} and \ref{thm:3}, in the second halves we describe the $\ell^p$ concentration of 
the heat kernel, namely, the regions $\scrN_n^p \subseteq V_g$ such that 
$\|k_n(o, \cdot)\|_{\ell^p(V_g \smallsetminus \scrN_n^p)}=o(\|k_n(o, \cdot)\|_{\ell^{p}})$.
This phenomenon reflects the interplay between the heat kernel decay on the one side and the volume growth on 
the other. Concerning the continuous counterparts of affine buildings, in \cite[Corollary 5.7.3]{Dav1989}
it was observed that the heat diffusion on real hyperbolic spaces propagates asymptotically with finite speed. This striking 
phenomenon was later on extended to non-compact symmetric spaces $\mathbb{X}$, by analytic means in \cite{AnSe1992} and 
probabilistic ones in \cite{B94}. Going beyond $L^1(\mathbb{X})$ and considering the general $L^p(\mathbb{X})$ setting has been
done in \cite{aj}. Theorem \ref{thm:A} is the building analogue of the latter.

Our next result concerns the asymptotic behavior of caloric functions. To be more precise, the caloric function $u$ with 
the initial datum $f \in \ell^p(V_g)$, $p \in [1, \infty]$, is the solution to
\begin{equation}
	\label{eq:86}
	\begin{aligned}
		u(n+1, x) &= A u(n, x) \\
		&= \sum_{y \in V_g} k(y, x) u(n, y) \\
		u(0, x) &= f(x), \qquad (n, x) \in \NN_0 \times V_g.
	\end{aligned}
\end{equation}
It is easy to see that the unique solution to \eqref{eq:86} is given by the formula
\[
	u(n, x) = \sum_{y \in V_g} k_n(y, x) f(y).
\]
Before we define the \emph{mass functions} we need some more notation: Let $\sigma(x, y)$ denote the vectorial distance between
$x \in V_g$ and $y \in V_g$. The maximal boundary $\Omega$ can be viewed as the set of all sectors tipped at $o$. For
any $x \in V$, we set $\Omega(o, x)$ to be the set of all sectors tipped at $o$ containing $x$. The maximal boundary 
carries the harmonic measure $\nu$ centered at $o$, such that $\nu(\Omega(o, x))$ depends only on the vectorial distance
$\sigma(o, x)$. Let $\chi_0$ be the \emph{fundamental character} of the coweight lattice, see \eqref{eq:10} for the definition.
Lastly, let $h(x, y; \omega)$ be the Busemann (horocycle) function for $x, y \in V_g$ and $\omega \in \Omega$.
With all the notations we can introduce the Macdonalds function
\[
	P_{\sigma(x,y)}(z) = \int_\Omega e^{\sprod{\eta + z}{h(x, y; \omega)}} \, \nu_x({\rm d} \omega), 
	\qquad z \in \mathfrak{a}_{\CC}.
\]
In order to define the mass function, let us introduce proper weighted spaces: We say that $f$ belongs to $\ell^1(w_p)$,
$p \in [1, \infty]$, if
\[
	\sum_{y \in V_g} |f(y)| w_p(y) < \infty
\]
where the weight $w_p$ is defined by the formula
\[
	w_p(y) =
		\begin{cases} \exp\big(\tfrac{2}{p}\sprod{\eta}{\sigma(o, y)}\big),  &\text{if } p \in [1, 2), \\
	\exp\big(\sprod{\eta}{\sigma(o, x)}\big), &\text{if } p \in [2, \infty].
	\end{cases}
\]
Now, for $f \in \ell^1(w_p)$ the mass function $M_p(f)$ is defined by the formula
\begin{align} 
	\label{eq:55}
	M_p(f)(x) &= \sum_{y \in V_g} f(y)	
	\fint_{\Omega(o,x)} \chi_{0}(h(o,y;\omega))^\frac{1}{p} \nu({\rm d} \omega), \qquad &\text{if } p \in [1, 2), \\
\intertext{or}
	\label{eq:56}
	M_p(f)(x) &= \frac{1}{P_{\sigma(o, x)}(0)} \sum_{y \in V_g} f(y) P_{\sigma(y, x)}(0), &\text{if } 
	p \in [2, \infty). 
\end{align}
We say that a function $f$ is radial if $f(x)$ depends only on $\sigma(o, x)$. In Section \ref{sec:6}, we show that the mass
function is well-defined for radial $f \in \ell^1(V_g)$, and
\begin{equation}
	\label{eq:87}
	M_p(f) \equiv \sum_{y\in V_g} f(y) P_{\sigma(o,y)}(s_p)
\end{equation}
where
\[
	s_p =
	\begin{cases}
		\eta\Big(\tfrac{2}{p} - 1 \Big), &\text{if } p \in [1, 2), \\
		0 &\text{if } p \in [2, \infty).
	\end{cases}
\]
The following theorem is a consequence of Theorems \ref{thm:11} and \ref{thm:10}. 
\begin{main_theorem}
	Let $k_n$ be a transition kernel of an admissible random walk on good vertices of an affine building $\scrX$.
	If $u$ is a solution to \eqref{eq:86} where $f \in \ell^1(w_p)$, $p \in [1, \infty]$ or $f$ is a radial and
	belongs to $\ell^1(V_g)$, then
	\[ 
		\lim_{n \to \infty} 
		\frac{1}{\|k_n(o, \cdot)\|_{\ell^p}}
		\| u(n;\cdot)-M_p(f)(\cdot)\, k_n(o,\cdot)\|_{\ell^p}=0.
	\]
\end{main_theorem}
The most important class of functions in $\ell^1(w_p)$ are whose that are finitely supported. For this class Theorem
\ref{thm:6} provides also the rate of convergence. In view of \eqref{eq:87}, for radial functions in $\ell^1(V_g)$
we have 
\[
	M_1(f) \equiv \sum_{y \in V_g} f(y),
\]
thus we obtain the discrete analogues of \cite{VazHyp}, \cite{APZ2023} and \cite{P1}. Let us also emphasize that we wish to cover 
situations where no sufficiently transitive group action is available which is a way to include some intriguing lower-dimensional 
exotic affine buildings (for existence see e.g. \cite{RonanExotic}), which is a deviation from the continuous counterpart.

In addition, let us stress the remarkably different behavior of the mass functions for $p \in [1, 2)$ and for 
$p \in (2, \infty]$. It is a result of the different regions of $\ell^p$ heat propagation, as described in Theorems \ref{thm:1} 
and \ref{thm:3}, respectively. In the case $p=2$ the exponential decay of $k_n(o, \cdot)^2$ cancels the exponential volume 
growth. To further highlight this phenomenon, we observe that the convergence result for $p=2$ holds 
not only for the mass function $M_2(f)$ defined by \eqref{eq:56}, but also for the mass function \eqref{eq:55}. It is worth pointing out that these two functions \emph{do not} necessarily coincide, unless $f$ is radial, whence they both give 
the constant 
\[
	M_2(f)=\sum_{y\in V_g} f(y) P_{\sigma(o,y)}(0).
\]
The heart of the proofs in this work is a careful use and analysis of asymptotic behavior of both the heat kernel itself and
its quotients in certain regions in the Cr{\'a}mer's zone. The asymptotic formulas for the transition densities $k_n(x,y)$ of 
finite range isotropic random walks on affine buildings were the subject of the recent article \cite{tr} of the second-named 
author. In the continuous setting, namely symmetric spaces of non-compact type, the large-time behavior of caloric functions 
was analyzed on the ``frequency'' side in \cite{APZ2023} and \cite{NRS24}; on real hyperbolic space \cite{VazHyp} the methods 
are reduced to arguments for functions of one variable, while heat kernel asymptotic behavior were used only recently in 
\cite{P1}. Nevertheless, let us point out certain challenges compared to the Archimedean setting. First of all, no group action 
is assumed throughout this text, so instead of group-theoretic tools employed in the continuous setting (such as the Iwasawa 
decomposition), one has to work in geometric language; compare for instance the mass functions for $p \in [1, 2)$ on affine
buildings in \eqref{eq:55} and in the Archimedean setting \cite{P1}. Moreover, there is no strictly Gaussian term in the used 
asymptotics \eqref{eq:3a} and \eqref{eq:3b}, and one cannot simply introduce a Gaussian term without also introducing an 
hard-to-control error as well. Finally, a difference in homogeneity is worth pointing out: the r\^ole of $\eta$ on affine 
buildings, say for instance on the definition of the $\ell^p(V_g)$ critical regions or of the mass functions for $p \in [1, 2)$, 
is played on symmetric spaces by $\rho$. We also introduce a Helgason transform on affine buildings including the exotic ones,
which is of independent interest, see Theorems \ref{thm:12} and Theorem \ref{thm:14}.

This paper is organized as follows: In Section \ref{sec:1} we present some preliminaries on affine buildings and the 
corresponding harmonic analysis, and random walks. In Section \ref{sec:2} we discuss the $\ell^p$ norms of heat kernels as well 
as $\ell^p$ heat concentration on critical regions, while Section \ref{sec:3} is devoted to obtaining asymptotics between ratios
of transition densities, as an indispensable tool for our work. Finally, in Section \ref{sec:4} we discuss the asymptotic 
behavior of caloric functions. We start by introducing the Helgason transform and its connection to the spherical transform 
for radial functions, then introduce mass functions for $p \in [1, \infty]$, and we discuss their properties. We next proceed 
to showing that a caloric function with a finitely supported initial datum converges in the $\ell^p$ critical region to the mass
function times the transition density. On the other hand, we show that both the solution and the heat kernel vanish
asymptotically outside that critical regions. In the last subsection, we discuss these problems for more general initial 
data in the $\ell^p$, $p\in [1, \infty]$ setting. In Appendix \ref{sec:8}, for convenience of the reader we provide
short proofs of Herz's \emph{principe de majoration} and Kunze--Stein phenomenon, for radial functions. 

\subsection*{Notation} Throughout this paper, we use the convention that $C, C', c, c',...$ stand for a generic positive 
constant whose value can change from line to line. The notation $A\lesssim B$ between two positive expressions means that 
there is a constant $C>0$ such that $A \leqslant C B$. The notation $A\approx B$ means that $A\lesssim B$ and $B\lesssim A$.
By $\NN$ we denote non-negative integers, whereas $\NN_0$ consists of positive integers.

\subsection*{Statements and Declarations}
The first named author acknowledges support by the Deutsche Forschungsgemeinschaft (DFG, German Research Foundation) 
SFB-Geschäftszeichen Projektnummer SFB-TRR 358/1 2023 491392403. This research was partially supported by the Polish National 
Science Centre Grant No. 2023/49/B/ST1/00678. 

\section{Affine buildings}
\label{sec:1}
	
\subsection{Buildings}
\label{sec:1.1}
A family $\scrX$ of non-empty finite subsets of some set $V$ is an \emph{abstract simplicial complex}
if for all $\sigma \in \scrX$, each subset $\gamma \subseteq \sigma$ also belongs to $\scrX$. The elements of $\scrX$
are called \emph{simplices}. The dimension of a simplex $\sigma$ is $\#\sigma - 1$. Zero dimensional simplices are called
\emph{vertices}. The set $V(\scrX) = \bigcup_{\sigma \in \scrX} \sigma$ is the \emph{vertex set} of $\scrX$.
The dimension of the complex $\scrX$ is the maximal dimension of its simplices. A \emph{face} of a simplex $\sigma$
is a non-empty subset $\gamma \subseteq \sigma$. For a simplex $\sigma$  we denote by $\St(\sigma)$ the collection of
simplices containing $\sigma$; in particular, $\St(\sigma)$ is a simplicial complex. Two abstract simplicial complexes
$\scrX$ and $\scrX'$ are \emph{isomorphic} if there is a bijection $\psi: V(\scrX) \rightarrow V(\scrX')$ such that for
all $\sigma = \{x_1, \ldots, x_k\} \in \scrX$ we have $\psi(\sigma) = \{\psi(x_1), \ldots, \psi(x_k)\} \in \scrX'$.
With every abstract simplicial complex $\scrX$ one can associate its \emph{geometric realization} $|\scrX|$ in the vector 
space of functions $V \rightarrow \RR$ with finite support, see e.g. \cite[\S2]{Munkres1996}.
	
A set $\scrC$ equipped with a collection of equivalence relations $\{\sim_i : i \in I\}$ where $I = \{0, \ldots, r\}$,
is called a \emph{chamber system} and the elements of $\scrC$ are called \emph{chambers}. A \emph{gallery} of type
$f = i_1 \ldots i_k$ in $\scrC$ is a sequence of chambers $(c_1, \ldots, c_k)$ such that for all
$j \in \{1,2, \ldots k\}$, we have $c_{j-1} \sim_{i_j} c_j$, and $c_{j-1} \neq c_j$. If $J \subseteq I$, a \emph{residue} of 
type $J$ is a subset of $\scrC$ such that any two chambers can be joined by a gallery of type $f = i_1 \ldots i_k$ with
$i_1, \ldots, i_k \in J$. From a chamber system $\scrC$ we can construct an abstract simplicial complex where each
residue of type $J$ corresponds to a simplex of dimension $r - \#J$. Then, for a given vertex $x$, we denote by
$\calC(x)$ the set of chambers containing $x$. 
	
A \emph{Coxeter group} is a group $W$ given by a presentation
\[
	\left\langle
	r_i : (r_i r_j)^{m_{i, j}} = 1, \text{ for all } i, j \in I
	\right\rangle
\]
where $M = (m_{i,j})_{I \times I}$ is a symmetric matrix with entries in $\ZZ \cup \{\infty\}$ such that for all
$i, j \in I$,
\[
	m_{i,j}=
	\begin{cases}
	\geqslant2 & \text{if }i\neq j,\\
	1 & \text{if }i=j.
	\end{cases}
\]
For a word $f=i_1 \cdots i_k$ in the free monoid $I$ we denote by $r_f$ an element of $W$ of the form
$r_f= r_{i_1} \cdots r_{i_k}$. The length of $w \in W$, denoted $\ell(w)$, is the smallest integer $k$ such that there
is a word $f=i_1 \cdots i_k$ and $w=r_f$. We say that $f$ is reduced if $\ell(r_f) = k$. A Coxeter group $W$ may be turned
into a chamber system by introducing in $W$ the following collection of equivalence relations:
$w \sim_i w'$ if and only if $w = w'$ or $w = w' r_i$. The corresponding simplicial complex $\Sigma$ is called 
\emph{Coxeter complex}. 
	
A simplicial complex $\scrX$ is called a \emph{building of type $\Sigma$} if it contains a family of subcomplexes called
\emph{apartments} such that
\begin{enumerate}[start=0, label=(B\roman*), ref=B\roman*]
	\item[(Bi)]\label{en:3:1}
	each apartment is isomorphic to $\Sigma$,
	\item[(Bii)]\label{en:3:2}
	any two simplices of $\scrX$ lie in a common apartment,
	\item[(Biii)]\label{en:3:3}
	for any two apartments $\scrA$ and $\scrA'$ having a chamber in common there is an 
	isomorphism $\psi: \scrA \rightarrow \scrA'$ fixing $\scrA \cap \scrA'$ pointwise.
\end{enumerate}
The rank of the building is the cardinality of the set $I$. We always assume that $\scrX$ is irreducible.
A simplex $c$ is a chamber in $\scrX$ if it is a chamber in any of its apartments. By $C(\scrX)$ we denote the set of chambers 
in $\scrX$. Using the building axioms we see that $C(\scrX)$ has a chamber system structure. However, it is not unique. 
A geometric realization of the building	$\scrX$ is its geometric realization as an abstract simplicial complex. In this article 
we assume that the system of apartments in $\scrX$ is \emph{complete}, meaning that any subcomplex of $\scrX$ isomorphic to 
$\Sigma$ is an apartment. We denote by $\Aut(\scrX)$ the group of automorphisms of the building $\scrX$.

\subsection{Affine Coxeter complexes}
\label{sec:1.2}
In this section we recall basic facts about root systems and Coxeter groups. A general reference is \cite{Bourbaki2002},
which deals with Coxeter systems attached to reduced root systems. Since we use from the beginning possibly non-reduced
root systems, we will also refer to \cite{mz1, park}.

Let $\Phi$ be an irreducible, but not necessarily reduced, finite root system in a Euclidean space $\mathfrak{a}$. 
Select a subset $\Phi^+$ of positive roots and let $\{\alpha_1,\dots,\alpha_r\}$ be the corresponding basis of simple roots.

For a positive root $\alpha \in \Phi^+$, its height is defined as
\[
	\height(\alpha)=\sum_{i=1}^rn_i 
\]
wherever $\alpha=\sum_{i=1}^rn_i\alpha_i$ with $n_i\in\NN$. Since $\Phi$ is irreducible, there is a unique highest root
$\alpha_0=\sum_{i=1}^rm_i\alpha_i$ with $m_i\in\NN_0$. We set
\[
	I_0=\{1,\dots,r\}
	\quad\text{and}\quad
	I_g=\{0\}\cup\{i\in I_0:m_i=1\}.
\]
Let $\Phi\spcheck=\{\alpha\spcheck =\frac2{\sprod{\alpha}{\alpha}}\alpha:\alpha\in\Phi\}$ be the dual root system.
The $\zspan$ of $\Phi\spcheck$ is the \emph{co-root lattice} $Q$, whose dual is the \emph{weight lattice} $L$.
Let $Q^+=\sum_{\alpha\in\Phi^+}\NN_0\alpha\spcheck$. The dual basis $\{\lambda_1,\dots,\lambda_r\}$ to 
$\{\alpha_1,\dots,\alpha_r\}$ consists of \emph{fundamental co-weights} and its $\zspan$ is the \emph{co-weight lattice} $P$.
Let $P^+=\sum_{i=1}^r\NN_0\lambda_i$ be the cone of \emph{dominant co-weights} and let $\sum_{i=1}^r\NN \lambda_i$ be 
the subcone of \emph{strongly dominant co-weights}. For instance
\[
	\rho=\sum_{i=1}^r\lambda_i
\]
is a strongly dominant co-weight. Notice that
\[
	\rho=\frac12\sum_{\alpha \in \Phi^{++}} \alpha\spcheck,
\]
where $\Phi^{++}$ denotes the set of indivisible positive roots in $\Phi$, that is the set of roots $\alpha \in \Phi^+$
such that $\frac{1}{2} \alpha \notin \Phi^+$.

For $\lambda \in P^+$, by $\Pi_{\lambda}$ we denote the saturation set of $\lambda$, that is
\[
	\Pi_\lambda= \{w.\mu:\mu\in P^+,\, \mu\preceq\lambda,\, w\in W\},
\]
whereas $\mu \preceq \lambda$, if and only if $\lambda - \mu\in Q^+$. If $\lambda, \mu \in P^+$ we write $\lambda \gg \mu$
whenever $\lambda - \Pi_\mu \subset P^+$. 

Let $\calH$ be the family of affine hyperplanes, called \emph{walls}, being of the form
\[
	H_{\alpha; k} = \big\{x \in \frakA : \langle x, \alpha \rangle = k \big\}
\]
where $\alpha \in \Phi^+$ and $k \in \ZZ$. Each wall determines two half-apartments
\[
	H^-_{\alpha; k} = \big\{x \in \frakA : \langle x, \alpha \rangle \leqslant k\big\}
	\quad\text{and}\quad
	H^+_{\alpha; k} = \big\{x \in \frakA : \langle x, \alpha \rangle \geqslant k\big\}.
\]
Note that for a given $\alpha$, the family $H^-_{\alpha; k}$ is increasing in $k$ while the family $H^+_{\alpha; k}$ is
decreasing. To each wall we associate $r_{\alpha; k}$ the orthogonal reflection in $\frakA$
defined by 
\[
	r_{\alpha; k}(x) = x - \big(\sprod{x}{\alpha} - k\big)\alpha\spcheck.
\]
Set $r_0 = r_{\alpha_0; 1}$, and $r_i = r_{\alpha_i; 0}$ for each $i \in I_0$. 
	
The \emph{finite Weyl group} $W$ is the subgroup of $\GL(\mathfrak{a})$ generated by $\{r_i: i \in I_0\}$. Let us denote
by $w_0$ the longest element in $W$. The \emph{fundamental sector} in $\frakA$ defined as
\[
	S_0 = \big\{x \in \frakA : \sprod{x}{\alpha_i} \geqslant 0 \text{ for all } i \in I_0\big\} 
	= \bigoplus_{i \in I_0} \RR_+ \lambda_i 
	= \bigcap_{i \in I_0} H^+_{\alpha_i; 0}
\]
is the fundamental domain for the action of $W$ on $\frakA$.
	
The \emph{affine Weyl group} $W^a$ is the subgroup of $\Aff(\mathfrak{a})$ generated by $\{r_i: i \in I\}$. Observe that 
$W^a$ is a Coxeter group. The hyperplanes $\calH$ give the geometric realization of its Coxeter complex $\Sigma_\Phi$. To see
this, let $C(\Sigma_\Phi)$ be the family of closures of the connected components of
$\frakA \setminus \bigcup_{H \in \calH} H$. By $C_0$ we denote the \emph{fundamental chamber} (or \emph{fundamental alcove}), 
i.e.
\[
	C_0 = \big\{x \in \frakA : \sprod{x}{\alpha_0} \leqslant 1 \text{ and } \sprod{x}{\alpha_i} \geqslant 0 \text{ for all } 
	i \in I_0\big\} = \bigcap_{i \in I_0} H^+_{\alpha_i; 0}  \cap H^-_{\alpha_0; 1}
\]
which is the fundamental domain for the action of $W^a$ on $\frakA$. Moreover, the group $W^a$ acts simply transitively 
on $C(\Sigma_\Phi)$. This allows us to introduce a chamber system in $C(\Sigma_\Phi)$: For two chambers $C$ and $C'$ and
$i \in I$, we set $C \sim_i C'$ if and only if $C = C'$ or there is $w \in W^a$ such that $C = w . C_0$ and
$C' = w r_i . C_0$. 
	
The vertices of $C_0$ are $\{0, \lambda_1/m_1, \ldots, \lambda_r/m_r\}$. Let us denote the set of vertices of all 
$C \in C(\Sigma_\Phi)$ by $V(\Sigma_\Phi)$. Under the action of $W^a$, the set $V(\Sigma_\Phi)$ is made up of $r+1$
orbits $W^a.0$ and $W^a.(\lambda_i/m_i)$ for all $i \in I_0$. Thus setting $\tau_{\Sigma_\Phi}(0) = 0$, and
$\tau_{\Sigma_\Phi}(\lambda_i/m_i) = i$ for $i \in I_0$, we obtain the unique labeling
$\tau_{\Sigma_\Phi} : V(\Sigma_\Phi) \rightarrow I$ such that any chamber $C \in C(\Sigma_\Phi)$ has one vertex with each label. 
	
For each simplicial automorphism $\vphi: \Sigma_\Phi \rightarrow \Sigma_\Phi$ there is a permutation $\pi$ of the set $I$
such that for all chambers $C$ and $C'$, we have $C \sim_i C'$ if and only if $\vphi(C) \sim_{\pi(i)} \vphi(C')$, and 
\[
	\tau_{\Sigma_\Phi}(\vphi(v)) = \pi(\tau_{\Sigma_\Phi}(v)),
	\qquad\text{for all } v \in V(\Sigma_\Phi).
\]
A vertex $v$ is called \emph{special} if for each $\alpha \in \Phi^+$ there is $k$ such that $v$ belongs to $H_{\alpha; k}$.
The set of all special vertices is denoted by $V_s(\Sigma_\Phi)$. 
	
Given $\lambda \in P$ and $w \in W^a$, the set $S = \lambda + w . S_0$ is called a \emph{sector} in $\Sigma_\Phi$ with a
\emph{base vertex} $\lambda$.
	
Moreover, by \cite[Corollary 3.20]{Abramenko2008}, an affine Coxeter complex $\Sigma_\Phi$ uniquely determines the affine Weyl
group $W^a$ but not a finite root system $\Phi$. In fact, the root systems $\text{C}_r$ and $\text{BC}_r$ have the same
affine Weyl group.
	
\subsection{Affine buildings}
\label{sec:1:3}
A building $\scrX$ of type $\Sigma$ is called an \emph{affine building} if $\Sigma$ is a Coxeter complex corresponding
to an affine Weyl group. Select a chamber $c_0$ in $C(\scrX)$ and an apartment $\scrA_0$ containing $c_0$. Using
an isomorphism $\psi_0: \scrA_0 \rightarrow \Sigma$ such that $\psi_0(c_0) = C_0$, we define the labeling in $\scrA_0$ by
\[
	\tau_{\scrA_0}(v) = \tau_\Sigma(\psi_0(v)), \qquad v \in V(\scrA_0).
\]
Now, thanks to the building axioms the labeling can be uniquely extended to $\tau: V(\scrX) \rightarrow I$. We turn 
$C(\scrX)$ into a chamber system over $I$ by declaring that two chambers $c$ and $c'$ are $i$-adjacent if they share all
vertices except the one of type $i$ (equivalently, they intersect along an $i$-panel). For each $c \in C(\scrX)$ and $i \in I$,
we define
\[
	q_i(c) = \#\big\{c' \in C(\scrX) : c' \sim_i c \big\} - 1.
\]
In all of the paper, we \emph{assume} that $q_i(c)$ only depends on $i$, i.e. that $q_i(c)$ is independent of $c$, and therefore 
the building $\scrX$ is \emph{regular}; we henceforth write $q_i$ instead of $q_i(c)$. We also assume that $1 < q_i(c) < \infty$
and therefore the building $\scrX$ is \emph{thick} and \emph{locally finite}. Notice that for Bruhat--Tits buildings this is 
automatic. A vertex of $\scrX$ is special if it is special in any of its apartments. The set of special vertices is denoted by 
$V_s$. We choose the finite root system $\Phi$ in such a way that $\Sigma$ is its Coxeter complex. In all cases except when the
affine group has type $\text{C}_r$ or $\text{BC}_r$, the choice is unique. In the remaining cases we select $\text{C}_r$
if $q_0 = q_r$, otherwise we take $\text{BC}_r$. This guarantees that $q_{\tau(\lambda)} = q_{\tau(\lambda+\lambda')}$
for all $\lambda, \lambda' \in P$, see the discussion in \cite[Section 2.13]{mz1}.

Let us define the set of \emph{good} vertices $V_g$ consisting of those $x \in V(\scrX)$ having the
type $\tau(x) \in I_g$. If $\Phi$ is reduced, then all special vertices are good. However, we do not have this property
in the case of $\text{BC}_r$. Indeed, there are two types of special vertices, $0$ and $r$, but only those
of type $0$ are mapped to $P$. 

For $\alpha \in \Phi^{++}$, we set $q_{\alpha} = q_{\alpha_j}$ provided that $\alpha \in W . \alpha_j$ for $j \in I$. We define
\[
	\tau_\alpha = 
	\begin{cases}
		1 	& \text{ if } \alpha \notin \Phi, \\
		q_\alpha & \text{ if } \alpha \in \Phi, \text{ but } \tfrac{1}{2} \alpha, 2 \alpha \notin \Phi, \\
		q_{\alpha_0} & \text{ if } \alpha, \tfrac{1}{2} \alpha \in \Phi, \text{ and therefore } 2 \alpha \notin \Phi, \\ 
		q_{\alpha} q^{-1}_{\alpha_0} & \text{ if } \alpha, 2 \alpha \in \Phi, \text{ and therefore } 
		\tfrac{1}{2}\alpha \notin \Phi.
	\end{cases}
\]
Notice that
\[
	\tau_{\alpha} \tau_{2\alpha}^2 > 1, \qquad \text{for all } \alpha \in \Phi^{++}.
\]
In this paper we often use the vector $\eta \in \mathfrak{a}$,
\begin{equation}
	\label{eq:6}
	\eta = \frac{1}{2} \sum_{\alpha \in \Phi^+} \log \tau_{\alpha}\,\alpha, \\
\end{equation}
The basic properties of $\eta$ are summarized in the following lemma.
\begin{lemma}
	\label{lem:3}
	We have
	\begin{equation}
		\label{eq:6'}
		\eta = \frac{1}{2} \sum_{\alpha \in \Phi^{++}} \log \big( \tau_{\alpha} \tau_{2 \alpha}^2\big) \, \alpha.
	\end{equation}
	Furthermore, 
	\begin{enumerate}[label=(\roman*), ref=\roman*]
		\item 
		\label{en:2:1}
		for each $\lambda \in P^+$ and $\alpha \in \Phi^{++}$, $\sprod{\alpha}{\lambda} \leqslant 4 \sprod{\eta}{\lambda}$.
		\item 
		\label{en:2:2}
		$w_0. \eta = - \eta$;
		\item 
		\label{en:2:3}
		for each $i \in I_0$, $\sprod{\eta}{\alpha_i} = \frac{1}{2} \log (\tau_{\alpha_i} \tau_{2\alpha_i}^2) 
		\sprod{\alpha_i}{\alpha_i}$;
		\item 
		\label{en:2:4}
		$\eta \in S_0$;
		\item
		\label{en:2:5}
		for each $w \in W$,
		$w.\eta = \eta + \sum_{\alpha \in w^{-1}. \Phi^{++} \cap (-\Phi^{++})} \log (\tau_\alpha \tau_{2\alpha}^2)
		\alpha$.
	\end{enumerate}
\end{lemma}
\begin{proof}
	The proof of \eqref{eq:6'} is straightforward. Since the building is thick, $q_\alpha \geqslant 2$, which leads to 
	\eqref{en:2:1}. For the proof of \eqref{en:2:2}, it is enough to observe that $w_0. \Phi^{++} = -\Phi^{++}$. 
	Since \eqref{en:2:3} easily leads to \eqref{en:2:4}, it remains to show \eqref{en:2:3}.
	If the root system is reduced, then each simple reflection $r_i$ sends $\alpha_i$ to $-\alpha_i$ and permutes
	other positive roots. Since $\tau_{r_i.\alpha} = \tau_{\alpha}$, we have
	\begin{align*}
		r_i.\eta 
		&= \frac{1}{2} \sum_{\stackrel{\alpha \in \Phi^+}{\alpha \neq \alpha_i}} 
		\log \tau_{\alpha} \alpha - \frac{1}{2} \log q_{\alpha_i} \, \alpha_i \\
		&=
		\eta - \log q_{\alpha_i}\, \alpha_i.
	\end{align*}
	Therefore,
	\begin{align*}
		\sprod{\eta}{\alpha_i} 
		&= -\sprod{r_i. \eta}{\alpha_i} \\
		&= -\sprod{\eta}{\alpha_i} + \log q_{\alpha_i} \sprod{\alpha_i}{\alpha_i},
	\end{align*}
	and so
	\[
		\sprod{\eta}{\alpha_i} = \frac{1}{2} \log q_{\alpha_i} \sprod{\alpha_i}{\alpha_i}.
	\]
	For the $\text{BC}_r$ root system it is enough to observe that the Weyl group $W$ is the same as for the root
	system $\Phi^{++} \sqcup (-\Phi^{++})$ which has type $\text{B}_r$. This proves \eqref{en:2:3}.

	Lastly, for the proof of \eqref{en:2:5} it is enough to recall that $w.\Phi = \Phi$.
\end{proof}
Next let us define a multiplicative function on $\mathfrak{a}$ by setting 
\begin{equation}
	\label{eq:10}
	\begin{aligned}
	\chi_0(x)
	&= \prod_{\alpha \in \Phi^+} \tau_\alpha^{\sprod{x}{\alpha}}\\
	&= e^{2 \langle x, \eta \rangle}, \qquad x \in \mathfrak{a}. 
	\end{aligned}
\end{equation}
For $w \in W^a$, having the reduced expression $w = r_{i_1} \cdots r_{i_k}$, we set $q_{w} = q_{i_1} \cdots q_{i_k}$. Then
we define
\begin{equation}
	\label{eq:16}
	N_\lambda=\frac{W(q^{-1})}{W_\lambda(q^{-1})}\chi_0(\lambda)
\end{equation}
where $W_\lambda=\bigl\{w\in W:w .\lambda=\lambda\bigr\}$ and whereas for any subset $U\subseteq W$, we set 
\[
	U(q^{-1})=\sum_{w\in U}q_w^{-1}.
\]
Given two special vertices $x, y \in V_s$, let $\scrA$ be an apartment containing $x$ and $y$ and let $\psi:
\scrA \rightarrow \Sigma$ be a type-rotating isomorphism such that $\psi(x) = 0$ and $\psi(y) \in S_0$, see
\cite[Definition 4.1.1]{park}. We set $\sigma(x, y) = \psi(y) \in \frac12 P^+$. If $x$ and $y$ are good vertices then
$\sigma(x, y) \in P^+$. For $\lambda \in P^+$ and $x \in V_g$, we denote by $V_{\lambda}(x)$ the set of all good vertices
$y \in V_g$ such that $\sigma(x, y) = \lambda$. The building axioms entail that the cardinality of $V_\lambda(x)$ depends
only on $\lambda$, see \cite[Proposition 1.5]{park2}, and it is equal to $N_\lambda$.

A subcomplex $\scrS$ is called a sector of $\scrX$ if it is a sector in any apartment. Two sectors are \emph{equivalent} if 
they contain a common subsector. The set of equivalence classes of sectors is denoted by $\Omega$. For any special vertex
$x \in V_s$ and $\omega \in \Omega$ there is a unique sector, denoted by $[x, \omega]$, which has base vertex $x$ and represents 
$\omega$. For $x \in V_s$ and $y \in V(\scrX)$, we set
\[
	\Omega(x, y) = \big\{\omega \in \Omega : y \in [x, \omega]\big\}.
\]
Given $x \in V_s$, the collection $\{\Omega(x, y) : y \in V_s\}$ generates a totally disconnected compact Hausdorff topology
on $\Omega$. We fix once and for all an origin $o$, which is a special vertex of the chamber $c_0$.

The horocycle (or Busemann function) $h: V_s \times V_s \times \Omega \rightarrow \tfrac{1}{2} P$ is defined as follows: for 
two special vertices $x$ and $y$, and $\omega \in \Omega$, we set
\[
	h(x, y; \omega) = \sigma(x, z) - \sigma(y, z)
\]
where $z$ is any special vertex belonging to $[x, \omega] \cap [y, \omega]$. The value $h(x, y; \omega)$ is independent of $z$,
see e.g. \cite[Proposition 3.3]{mz1}. In view of \cite[Lemma 3.13]{park2}, if $z \in V_s$ and 
$\sigma(x, z) \gg \sigma(x, y)$ then $\Omega(x, z) \subseteq \Omega(y, z)$. Hence,
\begin{equation}
	\label{eq:4}
	h(x, y; \omega) = \sigma(x, z) - \sigma(y, z), \quad \text{for all } \omega \in \Omega(x, z).
\end{equation}
Furthermore, for all special $x, y, z \in V_s$ and $\omega \in \Omega$, we have the cocycle relation
\begin{equation}
	\label{eq:52}
	h(x, y; \omega) = h(x, z; \omega) + h(z, y; \omega)
\end{equation}
Given $x, y \in V_g$, by the Kostant's convexity theorem, see \cite{Kostant} or \cite[Lemma 3.19]{park2}, if 
$\sigma(x, y) = \lambda$ then for each $\omega \in \Omega$, $h(x, y; \omega) \in \Pi_\lambda$. Since 
$\sigma(x, y) - h(x, y; \omega) \in Q^+$, by Lemma \ref{lem:3}\eqref{en:2:3}, we have
\begin{equation}
	\label{eq:17}
	\sprod{h(x, y; \omega)}{\eta} 
	\leqslant \sprod{\sigma(x, y)}{\eta}.
\end{equation}
In fact, the formula $\sprod{\sigma(x, y)}{\eta}$ defines a distance on special vertices $V_s$. Indeed, for any three vertices 
$x, y, z \in V_s$, if $\omega \in \Omega(x, z)$, then by \eqref{eq:52} together with \eqref{eq:17}, we get
\begin{align}
	\nonumber
	\sprod{\sigma(x, z)}{\eta}
	&= \sprod{h(x, y; \omega)}{\eta} + \sprod{h(y, z; \omega)}{\eta} \\
	\label{eq:12}
	&\leqslant \sprod{\sigma(x, y)}{\eta} + \sprod{\sigma(y, z)}{\eta}.
\end{align}
In view of Lemma \ref{lem:3}\eqref{en:2:2}, we have
\[
	\sprod{\sigma(y, x)}{\eta} = -\sprod{w_0 . \sigma(x, y)}{\eta} =
	\sprod{\sigma(x, y)}{\eta}.
\]
Lastly, $\sprod{\sigma(x, y)}{\eta} = 0$ implies that $\sigma(x, y) = 0$ and so $x = y$. Analogously, one can show that
$|\sigma(x, y)|$ is a distance on special vertices $V_s$. In particular, for all $x, y, z \in V_s$,
\begin{equation}
	\label{eq:83}
	|\sigma(x, y)| \leqslant |\sigma(x, z)| + |\sigma(z, y)|.
\end{equation}
In view of \cite[Lemma 3.1]{RemyTrojan}, for each $x, y, z \in V_s$
\begin{equation}
	\label{eq:14}
	|\sigma(x, y) -\sigma(x, z)| \leq |\sigma(y, z)|,
\end{equation}
thus for $\alpha \in \Phi^+$, 
\begin{equation}
	\label{eq:11}
	|\sprod{\sigma(x, y)}{\alpha} - \sprod{\sigma(x, z)}{\alpha}| \leq |\alpha| |\sigma(y, z)|.
\end{equation}
For every special vertex $x$ there is a unique Borel probability measure $\nu_x$ on $\Omega$, such that
\[
	\nu_x\big(\Omega(x, y)\big) = \frac{1}{\#\{y' \in V_s : \sigma(x, y') = \sigma(x, y')\}}.
\]
The measures $\nu_x$ and $\nu_y$ for any $x, y \in V_s$ are mutually absolutely continuous, with the Radon--Nikodym derivative 
equal to
\begin{equation}
	\label{eq:15}
	\frac{{\rm d} \nu_y}{\mathrm{d} \nu_x}(\omega)
	=\chi_0(h(x, y; \omega)), \quad \omega \in \Omega.
\end{equation}
For details see \cite[page 16]{park2}, see also \cite[Proposition 6.1]{RemyTrojan}. We often write $\nu = \nu_o$.

\subsection{Macdonald spherical functions}
In this subsection we recall the definition and properties of Macdonald spherical functions (see \cite{macdo0}): 
for given $\lambda \in P^+$, 
\[
	P_\lambda(z) = \frac{\chi_0(\lambda)^{-\frac{1}{2}} }{W(q^{-1})}
	\sum_{w \in W} \bfc(w \cdot z) e^{\sprod{w \cdot z}{\lambda}}, \quad z \in \mathfrak{a}_{\mathbb{C}},
\]
where
\begin{align*}
	\bfc(z)
	&= \prod_{\alpha \in \Phi^+} \frac{1 - \tau_\alpha^{-1} \tau_{\alpha/2}^{-\frac{1}{2}} e^{-\sprod{z}{\alpha\spcheck}}}
	{1-\tau_{\alpha/2}^{-\frac{1}{2}} e^{-\sprod{z}{\alpha\spcheck}}} \\
	&=
	\prod_{\alpha \in \Phi^{++}}
	\frac{\Big(1 - \tau_{2\alpha}^{-1} \tau_{\alpha}^{-\frac{1}{2}} e^{-\frac{1}{2}\sprod{z}{\alpha\spcheck}}\Big)
		\Big(1 + \tau_{\alpha}^{-\frac{1}{2}} e^{-\frac{1}{2} \sprod{z}{\alpha\spcheck}}\Big)}
	{1 - e^{-{\sprod{z}{\alpha\spcheck}}}}, 
\end{align*}
and
\[
	W(q^{-1})=\sum_{w \in W} q_w^{-1}.
\]
Values of $P_\lambda$ where the denominator of the $\bfc$-function equals zero are obtained by taking proper limits. 
In view of \cite[formula (6.6)]{park2}, Macdonald spherical functions can be written as integrals over $\Omega$,
namely, for $\lambda \in P^+$ and $z \in \mathfrak{a}_{\CC}$,
\begin{equation}
	\label{eq:9}
	P_\lambda(z)=
	\int_{\Omega} \chi_0(h(x, y; \omega))^{\frac{1}{2}} \, e^{\langle z, h(x, y ; \omega)\rangle} \,\nu_x({\rm d} \omega)
\end{equation}
where $x$ and $y$ is any pair of special vertices such that $\sigma(x, y) = \lambda$.

Macdonald functions are the main tools to study the algebra of averaging operators on $V_g$. For each $\lambda \in P^+$,
we define an operator acting on finitely supported function of good vertices $V_g$ as
\[
	A_\lambda f(x)=\frac1{N_\lambda}\sum_{y\in V_\lambda(x)}f(y), \quad x \in V_g. 
\]
Then $\mathscr{A}_0=\cspan\{A_\lambda:\lambda\in P^+\}$ is a commutative unital algebra, whose characters can be expressed 
in terms of Macdonald spherical functions. Specifically, every multiplicative functional on 
$\mathscr{A}_0$ is given by the evaluation
\[
	h_z(A_\lambda)=P_\lambda(z), \quad \lambda\in P^+
\]
at some $z\in\mathfrak{a}_\mathbb{C}$. Moreover, $h_z=h_{z'}$ if and only if $W.z+i 2\pi L=W.z'+i2\pi L$ 
(see \cite[Theorem 3.3.12(ii)]{macdo0}). Let $\mathscr{A}_2$ be the closure of $\mathscr{A}_0$ in the operator norm on 
$\ell^2(V_g)$. Then $\mathscr{A}_2$ is a commutative unital $C^\star$-algebra. 

\vspace*{1ex}
\noindent
\emph{The standard case.} Assume that $\tau_\alpha \geqslant 1$ for all $\alpha \in \Phi$. Then for all $\theta \in U_0$,
where
\[
	U_0 = \big\{\theta \in \mathfrak{a} : \sprod{\theta}{\alpha\spcheck} \leqslant \pi \text{ for all } \alpha \in \Phi\big\}
\]
the multiplicative functional $h_{i\theta}$ extends to $\scrA_2$ in a continuous way. We set $\scrD = U_0$, and for any
any Borel subset $B \subset \scrD$,
\begin{equation}
	\label{eq:80}
	\pi(B) = 
	\bigg(\frac{1}{2\pi} \bigg)^r
	\frac{W(q^{-1})}{|W|} \int_B \frac{{\rm d} \theta}{|\bfc(i\theta)|^2}.
\end{equation}

\vspace*{1ex}
\noindent
\emph{The exceptional case.}
Suppose that $\tau_\alpha < 1$ for some $\alpha \in \Phi$, which is only possible when $\Phi$ is the $\text{BC}_r$ root system
and $q_r < q_0$, namely
\[
	\text{BC}_r = \big\{\pm e_i, \pm 2 e_i, \pm e_j \pm e_k : 1 \leqslant i \leq r, 1 \leqslant j < k \leqslant r\big\}
\]
where $\{e_1, e_2, \ldots, e_r\}$ is the standard basis of $\mathfrak{a}$. We set $a = \sqrt{q_r q_0}$ and $b= \sqrt{q_r/q_0}$.
Then
\[
	\bfc(z) = \bigg(\prod_{j = 1}^r \frac{(1-a^{-1} e^{-z_j})(1 + b^{-1} e^{-j})}{1 - e^{-2z_j}}\bigg)
	\bigg(\prod_{1 \leqslant j < k \leqslant r} \frac{(1 - q_1^{-1} e^{-z_j-z_k})(1 - q_1^{-1} e^{-z_j + z_k})}
	{(1-e^{-z_j-z_k})(1-e^{-z_j+z_k})}\bigg).
\]
Let $v = \log b + i \pi$. For $j = 1, \ldots, r$ we set
\[
	U_j = \big\{z \in \mathfrak{a}_\CC : z_k \in [-\pi, \pi] \text{ for all } k \neq j, 
	\text{ and } z_j = \pi - i\log b \big\},
\]
and $U_0 = [-\pi, \pi]^r$. For $\theta \in U_j$, we define
\[
	\phi_j(i\theta) = \bfc(-i\theta) \lim_{t \to 0} \frac{\bfc(i\theta+te_j)}{1 - e^t}.
\]
Then for each $\theta \in \scrD = U_0 \bigsqcup U_1$, the multiplicative functional $h_{i\theta}$ extends to $\scrA_2$ in 
a continuous way. For a Borel subset $B \subset \scrD$, we put
\begin{equation}
	\label{eq:81}
	\pi(B)
	= 
	\bigg(\frac{1}{2\pi}\bigg)^r
	\frac{W(q^{-1})}{|W|} \int_{B \cap U_0} \frac{{\rm d} \theta}{|\bfc(i\theta)|^2} 
	+
	\bigg(\frac{1}{2\pi}\bigg)^{r-1}
	\frac{W(q^{-1})}{|W'|} \int_{B \cap U_1} \frac{{\rm d} \theta}{\phi_1(i\theta)}
\end{equation}
where $W'$ is the Coxeter group $C_{r-1}$ and the measure ${\rm d} \theta$ on $U_j$ equals
\[
	{\rm d} \theta = \prod_{\stackrel{k = 1}{k \neq j}}^r {\rm d} \theta_k
\]
for $j = 0, \ldots, r$. Then for all $A \in \scrA_0$, $x \in V_g$ and $y \in V_\lambda(x)$, we have
\begin{equation}
	\label{eq:21}
	(A \delta_x)(y) 
	=
	\int_{\scrD} h_{i\theta}(A) \overline{P_\lambda(i \theta)} \pi({\rm d} \theta),
\end{equation}
see \cite[Theorem 5.2 \& Corollary 5.5]{park2} in the standard case, and
\cite[Theorem 5.7 \& Corollary 5.8]{park2} in the exceptional case.

A complex-valued function $\vphi: V_g \rightarrow \CC$, is called \emph{spherical} with respect to $x \in V_g$, if
\begin{enumerate}
	\item $\vphi(x) = 1$;
	\item $\vphi(y) = \vphi(y')$ for all $y, y' \in V_\lambda(x)$ for any $\lambda \in P^+$;
	\item for each $\lambda \in P^+$, there is $C_\lambda \in \CC$, such that $A_\lambda \vphi = C_\lambda \vphi$.
\end{enumerate}
Thanks to \cite[Theorem 3.22]{park2}, for each $z \in \mathfrak{a}_{\CC}$ the function $\vphi_z$ defined as
\begin{equation}
	\label{eq:60}
	\vphi_z(y) = P_\lambda(z), \qquad\text{if } y \in V_\lambda(o)
\end{equation}
is spherical with respect to $o$. In the following proposition we describe the asymptotic behavior of the 
\emph{ground state spherical function}, that is
\begin{equation}
	\label{eq:51}
	\mathbf{\Phi}(\lambda) = P_\lambda(0).
\end{equation}
The statement as well as the proof of the following proposition are similar to the symmetric space case \cite{a1}, for details
see \cite[Proposition 2.1]{ascht} where $\tilde{\text{A}}_n$ buildings have been considered. The latter proof generalizes 
in a straightforward way. For a more detailed description of the asymptotic behavior of the ground state spherical function 
we refer to \cite[Section 7.2]{RemyTrojan}.
\begin{proposition} 
	\label{prop:1}
	We have
	\[
		\mathbf{\Phi}(\lambda)
		\approx \chi_0(\lambda)^{-\frac{1}{2}} \prod_{\alpha \in \Phi^{++}}(1+\langle\alpha, \lambda\rangle).
	\]
	Furthermore,
	\[
		\mathbf{\Phi}(\lambda)
		= \text{ const. } \chi_0(\lambda)^{-\frac{1}{2}}  
		\prod_{\alpha \in \Phi^{++}} \langle\alpha, \lambda\rangle + o(1)
	\]
	as $\langle\alpha, \lambda\rangle \rightarrow +\infty$  for all  $\alpha \in \Phi^{+}.$
\end{proposition}

\subsection{Random walks}
In this paper we study caloric functions associated with an \emph{isotropic} random walk on good vertices $V_g$, i.e. the random 
walk with the transition probability $k(x, y)$ constant on 
\[
	\left\{(x, y) \in V_g \times V_g: y \in V_\lambda(x)\right\}
\]
for every $\lambda \in P^+$. Let $A$ denote the corresponding operator acting on compactly supported functions on $V_g$,
that is
\[
	A f(x) = \sum_{y \in V_g} k(x, y) f(y).
\]
Then $A$ belongs to the algebra $\mathscr{A}$ and can be expressed as
\begin{equation}
	\label{eq:8}
	A = \sum_{\lambda \in P^+} c_\lambda A_\lambda
\end{equation}
where $c_\lambda \geqslant 0$ and $\sum_{\lambda \in P^+} c_\lambda = 1$. We say that the random walk has a \emph{finite range} 
if $c_{\lambda}>0$ for finitely many $\lambda \in P^{+}$. We set $k(1 ; x, y)=k(x, y)$, and for $n \geqslant 2$,
\[
	k(n ; x, y)=\sum_{z \in V_s} k(n-1 ; x, z) k(z, y) .
\]
Let
\[
	k_n(x) = k(n; x) = k(n; o, x), \qquad x \in V_g, n \in \NN_0.
\]
The random walk is \emph{irreducible} if for any $x, y \in V_g$, there is $n \in \mathbb{N}$ such that
\[
	k(n ; x, y)>0 .
\]
Lastly, the walk is called \emph{aperiodic} if for every $x \in V_g$,
\[
	\operatorname{gcd}\{n \in \mathbb{N}: k(n ; x, x)>0\}=1
\]
We shall be concerned with irreducible aperiodic random walks having a finite range, which we call for short
\emph{admissible} random walks. Let us fix an admissible random walk with the transition function $k$. To describe the long time
asymptotic behavior of the corresponding heat kernel $k(n; x, y)$, we need to introduce some additional notation. Let 
$\varrho = h_0(A)$, and set
\[
	\kappa(z) = \varrho^{-1}  h_z(A), \quad z \in \mathfrak{a}_{\mathbb{C}}.
\]
In view of the Gelfand theorem, and the representation \eqref{eq:8} together with \cite[Theorem 6.5]{park2}, one can show that
$\varrho$ is the spectral radius of $A$ on $\ell^2(V_g)$. There exist a finite set $\mathcal{V} \subset P$ and positive real 
numbers $\left\{c_v: v \in \mathcal{V}\right\}$ with $\sum_{v \in \mathcal{V}} c_v = 1$ such that
\[
	\kappa(z)=\sum_{v \in \mathcal{V}} c_v e^{\langle z, v\rangle}.
\]
For $x \in \mathfrak{a}$, by $B_x$ we denote a quadratic form 
\begin{equation}
	\label{eq:7}
	B_x(u, u)=D_u^2 \log \kappa(x)
\end{equation}
where $D_u$ is the derivative along a vector $u$, i.e.,
\[
	D_u f(x)=\left.\frac{\mathrm{d}}{{\rm d} t} f(x+t u)\right|_{t=0}.
\]
Hence,
\[
	B_x(u, u)=\frac{1}{2} \sum_{v, v^{\prime} \in \mathcal{V}} 
	\frac{c_v e^{\langle x, v\rangle}}{\kappa(x)} \cdot \frac{c_{v^{\prime}} 
	e^{\left\langle x, v^{\prime}\right\rangle}}{\kappa(x)}\left\langle u, v-v^{\prime}\right\rangle^2.
\]
Let $\mathcal{M}$ be the interior of the convex hull of $\mathcal{V}$. Then for every $\delta \in \mathcal{M}$ 
a function $f(\delta, \cdot): \mathfrak{a} \rightarrow \mathbb{R}$ defined by
\[
	f(\delta, x)=\langle x, \delta\rangle-\log \kappa(x)
\]
attains its maximum at the unique point $s = s(\delta) \in \mathfrak{a}$ satisfying $\nabla \log \kappa(s)=\delta$,
see e.g. \cite[Theorem 2.1]{tr}. By the implicit function theorem, the function $s$ is real-analytic on 
$\mathcal{M}$. In particular, $s$ is bounded on any compact subset of $\mathcal{M}$. On the other hand, $|s(\delta)|$ approaches 
infinity when $\delta$ tends approaches $\partial\mathcal{M}$. Since $\kappa$ is $W_0$ invariant, for each $\alpha \in \Phi^+$,
$\sprod{\delta}{\alpha} = 0$ if and only if $\sprod{s}{\alpha} = 0$. Since $\sprod{\delta}{u} = D_u \log \kappa(s)$, we 
have
\[
	D_u s = B_s^{-1} u.
\]
Let us set
\begin{equation}
	\label{eq:20}
	\begin{aligned}
	\phi(\delta)
	&=
	f(\delta, s) \\
	&=
	\max\big\{\sprod{u}{\delta} - \log \kappa(u) : u \in \mathfrak{a} \big\}.
	\end{aligned}
\end{equation}
Then (see \cite[Section 2.1]{tr})
\begin{equation}
	\label{eq:33a}
	\phi(\delta) = \frac{1}{2} B_0^{-1}(\delta, \delta) + \calO(\abs{\delta}^3),
\end{equation}
and
\begin{equation}
	\label{eq:33b}
	\phi(\delta) \approx \norm{\delta}^2, \qquad \delta \in \calM.
\end{equation}
Let us summarize the properties of the heat kernel $k(n; x)$: In view of \cite[Remark 2]{tr}, there is $C > 0$
such that for all $n \in \NN$, and $x \in V_g$,
\begin{equation}
	\label{eq:3}
	k(n; x) \leqslant C \chi_0(\sigma(o, x) )^{-\frac{1}{2}} \varrho^n e^{-n \phi(\delta)}.
\end{equation}
where $\delta = \frac{\sigma(o, x)}{n}$. In this paper we need the precise asymptotic behavior of $k(n; x)$ away from the walls
of the Weyl cone $\mathfrak{a}_+$ while $\sigma(o, x)/n$ stays at bounded distance from the boundary of $\calM$. For this 
reason, we take $J = \emptyset$ in \cite[Theorem 4.1]{tr}. Then given $\xi > 0$ and $\varepsilon > 0$ for any sequence 
$(x_n : n \in \NN)$ of good vertices such that $k(n; x_n) > 0$ for all $n \in \NN$, with $\delta_n = n^{-1} \sigma(o, x_n)$ 
satisfying
\[
	\sprod{\delta_n}{\alpha} \geqslant \xi, \qquad \text{for all } \alpha \in \Phi^+,
\]
and $\dist(\delta_n, \partial \calM) \geqslant \varepsilon$, we have
\begin{equation}
	\label{eq:3a}
	k(n;x_n)
	=
	n^{-\frac{r}{2}} \varrho^n \,e^{-n\phi(\delta)}\, 
	\chi_0(\sigma(o,x_n))^{-\frac{1}{2}}\,( \det B_{s_n})^{-\frac{1}{2}} \, \frac{1}{\bfc(s_n)}\, 
	\Big(C_0 + \calO(n^{-1})\Big)
\end{equation}
with $s_n = \nabla \phi(\delta_n)$. The implicit constant in the error term depends only on $\xi$, $\varepsilon$, and $A$.

We also need the description of the asymptotic behavior of $k(n; y, x)$ in the region where $\sigma(y, x_n) = o(n)$.
By \cite[Corollary 4.10]{tr}, for each sequence $(x_n: n \in \NN)$ of good vertices such $k(n; y, x_n) > 0$ for all $n \in \NN$,
with $\delta_n = n^{-1} \sigma(y, x_n)$ satisfying
\[
	\lim _{n \rightarrow \infty} \left\langle \delta_n, \alpha\right\rangle=0, \quad \text {for all } \alpha \in \Phi,
\]
we have
\begin{equation}
	\label{eq:3b}
	k(n; y, x_n) =
	n^{-\frac{r}{2}-\left|\Phi^{++}\right|} \varrho^n e^{-n \phi\left(\delta_n\right)} 
	\mathbf{\Phi}(\sigma(y, x_n))
	\left(C_0+\mathcal{O}(\left|\delta_n\right|)+\mathcal{O}(n^{-1})\right).
\end{equation}
The constant $C_0$ is absolute.
	
\section{Norm estimates of heat kernels}
\label{sec:2}
We start our study by investigating the behavior of $\ell^p(V_g)$ norms, $p \in [1, \infty]$, of the heat kernel corresponding
to an isotropic finite range random walk. Let us define
\begin{equation}
	\label{eq:5}
	s_p = 
	\begin{cases}
		\eta \big(\frac{2}{p}-1\big) & \text{if } p \in [1, 2), \\
		0 & \text{otherwise,}
	\end{cases}
\end{equation}
and let
\[
	\delta_p = \nabla \log \kappa(s_p) \in \inter \calM.
\]
\begin{lemma}
	\label{lem:2}
	For all $\lambda \in P^+$ and $z \in [-1, 1]\eta + i U_0$, we have
	\[
		0 < P_\lambda(\Re z) \leqslant P_\lambda(\eta) = 1
	\]
	In particular, if $p \in (1, 2]$, then $\kappa(s_p)<\kappa(s_1)=\varrho^{-1}$.
\end{lemma}
\begin{proof}
	In view of \cite[Theorem 7]{tr}, if $z \in t \eta + U_0$, then $P_\lambda(\Re z) \leqslant P_\lambda(t \eta)$.
	Now, to prove the inequality, it suffices to show that the function
	\[
		[0, 1] \ni t \mapsto \log\kappa(t\eta),
	\]
	is strictly increasing. In view of \eqref{eq:7}, we have
	\[
		\frac{{\rm d} ^2}{{\rm d} t^2}\log \kappa(t\eta) = B_{t\eta}(\eta, \eta) > 0,
	\]
	which implies that the function 
	\[
		[0, 1] \ni t \mapsto \frac{{\rm d}}{{\rm d} t} \log \kappa(t \eta)=\langle \nabla \log \kappa(t\eta),\eta \rangle
	\]
	is strictly increasing. Since $\kappa$ is $W$-invariant, we have $\nabla \log \kappa(0) = 0$, which proves our claim.
	
	Let us now compute $\kappa(s_1)$. Since $s_1 = \eta$, by Lemma \ref{lem:3}\eqref{en:2:2} and the integral representation of
	Macdonald function \eqref{eq:9}, we get
	\begin{align*}
		P_{\lambda}(s_1) = P_{\lambda}(-\eta)
		&=\int_{\Omega}\chi_0^{\frac{1}{2}}(h(o,x;\omega))\, e^{-\langle \eta,h(o,x;\omega) \rangle}\, \nu({\rm d} \omega)  \\
		&=\int_{\Omega}\chi_0^{\frac{1}{2}}(h(o,x;\omega))\, \chi_0^{-\frac{1}{2}}(h(o,x;\omega))\, \nu({\rm d} \omega)=1.
	\end{align*}
	Hence, by the representation \eqref{eq:8}, we obtain
	\begin{align*}
		\varrho \kappa(\eta) = 
		\hat{A} (\eta) &=\sum_{\lambda \in P^{+}} c_{\lambda} \hat{A}_{\lambda} (\eta) \\
		&=\sum_{\lambda \in P^{+}} c_{\lambda} P_{\lambda}(\eta)
		=\sum_{\lambda \in P^{+}} c_{\lambda} = 1,
	\end{align*}
	and the lemma follows.
\end{proof}
\begin{lemma}
	\label{lem:4}
	Let $k$ be a transition kernel of an admissible random walk on good vertices of an affine building $\scrX$.
	Let $p \in [1, \infty)$. For all $b > a > 0$, there is $C > 0$ such that for each $n \in \NN$,
	\[
		\sum_{x \in V_g : a\leqslant |\sigma(o,x)|\leqslant b} k(n; x)^p
		\leqslant
		C
		\varrho^{np} \sum_{\lambda \in P^+ : a\leqslant |\lambda|\leqslant b}
		\exp\Big\{p n (\sprod{s_p}{\lambda/n} - \phi(\lambda/n))\Big\}.
	\]
\end{lemma}
\begin{proof}
	Using the global upper bound \eqref{eq:3}, we obtain 
	\begin{align*}
		\nonumber
		\sum_{x \in V_g : a\leqslant |\sigma(o,x)|\leqslant b}
		k(n; x)^p
		&\lesssim 
		\varrho^{np} \sum_{\lambda \in P^+ : a \leqslant |\lambda|\leqslant b} 
		\chi_0(\lambda )^{-\frac{p}{2}}  e^{-pn \phi(\lambda/n)} N_\lambda 
		\\
		&\lesssim
		\varrho^{np} \sum_{\lambda \in P^+ : a\leqslant |\lambda|\leqslant b}  
		e^{p n (\sprod{s_p}{\lambda/n} - \phi(\lambda/n))}
	\end{align*}
	where the last inequality follows by \eqref{eq:16} and \eqref{eq:5}. This completes the proof.
\end{proof}

\begin{lemma}
	\label{lem:1}
	Let $c > 0$, $\gamma > 0$, $\beta \in \NN^r$ and $x_0 \in \mathfrak{a}$. For each $n \in \NN_0$ we set
	\[
		\Lambda_n = \Big\{\lambda \in P^+ : \norm{\lambda - n x_0} \leqslant n^{\frac{1}{2}+\gamma}\Big\}.
	\]
	Then
	\begin{equation}
		\label{eq:23}
		\sum_{\lambda \in \Lambda_n}
		\Big(\prod_{j = 1}^r |\sprod{\lambda - n x_0}{\alpha_j}|^{\beta_j} \Big)
		e^{-c \frac{\norm{\lambda - n x_0}^2}{n} }
		\approx n^{\frac{r+|\beta|}{2}},
	\end{equation}
	and
	\begin{equation}
		\label{eq:24}
		\sum_{\stackrel{\lambda \in \Lambda_n}{|\lambda - n x_0|_\infty \geqslant n^{1/2}}}
		\Big(\prod_{j = 1}^r |\sprod{\lambda - n x_0}{\alpha_j}|^{\beta_j} \Big)
		e^{-c \frac{\norm{\lambda - n x_0}^2}{n} }
		\approx n^{\frac{r+|\beta|}{2}}
	\end{equation}
	as $n$ tends to infinity uniformly with respect to $x_0 \in \mathfrak{a}$.
\end{lemma}
\begin{proof}
	Let us observe that
	\[
		\Big\{\lambda \in P^+ : \norm{\lambda - n x_0}_{\infty} \leqslant r^{-\frac{1}{2}} n^{\frac{1}{2}+\gamma} \Big\}
		\subset
		\Lambda_n
		\subset
		\Big\{\lambda \in P^+ : \norm{\lambda - n x_0}_{\infty} \leqslant n^{\frac{1}{2}+\gamma} \Big\}.
	\]
	Hence, it is enough to show that for each $a > 0$ and $\beta \in \NN^r$,
	\[
		\sum_{\lambda \in P^+ : \norm{\lambda-n x_0}_{\infty} \leqslant an^{\frac{1}{2}+\gamma} } 
		\Big(\prod_{j = 1}^r |\sprod{\lambda - n x_0}{\alpha_j}|^{\beta_j} \Big)
		e^{-c \frac{\norm{\lambda-n x_0}^2}{n}}
		\approx 
		n^{\frac{r+|\beta|}{2}}.
	\]
	Observe that it is sufficient to consider $r = 1$, as the general case then easily follows. Since there is $C > 0$ such that
	for all $\tau \in [0, 1]$ and $j \in \NN$ satisfying $|j - n x_0| \leqslant a n^{1/2 + \gamma}$,
	\begin{align*}
		\Big|
		\exp\Big\{-c \frac{(j-n x_0)^2}{n}\Big\} - \exp\Big\{-c \frac{(j-\tau-n x_0)^2}{n}\Big\}
		\Big|
		\leqslant
		C n^{-\frac12+\gamma} \exp\Big\{-c \frac{(j-\tau-n x_0)^2}{n}\Big\},
	\end{align*}
	we obtain
	\begin{align*}
		&\Bigg|
		\sum_{j \in \NN : |j - n x_0| \leqslant a n^{\frac12+\gamma}} 
		\norm{j - nx_0}^b e^{-c \frac{(j-n x_0)^2}{n}}
		-
		\int_{\lfloor n x_0 - a n^{1/2+\gamma}\rfloor - 1}^{\lfloor n x_0 + a n^{\frac12+\gamma} \rfloor}
		|x-nx_0|^b e^{-c \frac{(x-n x_0)^2}{n}} {\: \rm d} x
		\Bigg| \\
		&\qquad\qquad\leqslant
		\sum_{j \in \NN : |j - n x_0| \leqslant a n^{\frac12+\gamma}}
		\int_0^1 
		|j - \tau -nx_0|^b
		\Big| e^{-c \frac{(j - n x_0)^2}{n}} - e^{-c \frac{(j-\tau-n x_0)^2}{n}}\Big| 
		{\: \rm d} \tau \\
		&\qquad\qquad\phantom{\leqslant \sum_{j \in \NN : |j - n x_0| \leqslant a n^{\frac12+\gamma}}}
		+
		\int_0^1 
		\big(|j - nx_0|^b - |j - \tau -nx_0|^b\big) e^{-c \frac{(j-\tau-n x_0)^2}{n}}
		{\: \rm d} \tau 
		\\
		&\qquad\qquad\leqslant
		C n^{-\frac12+\gamma} 
		\int_{\lfloor nx_0 - a n^{\frac12+\gamma}\rfloor-1}^{\lfloor nx_0 + a n^{\frac12+\gamma}\rfloor}
		|x-nx_0|^b e^{-c \frac{(x-nx_0)^2}{n}} {\: \rm d} x \\
		&\qquad\qquad\phantom{\leqslant}
		+ C
		\int_{\lfloor nx_0 - a n^{\frac12+\gamma}\rfloor-1}^{\lfloor nx_0 + a n^{\frac12+\gamma}\rfloor}
		|x - n x_0|^{b-1} e^{-c \frac{(x-nx_0)^2}{n}} {\: \rm d} x.
	\end{align*}
	Next, we write
	\begin{align*}
		\int_{nx_0 - a n^{\frac12+\gamma}}^{nx_0 + a n^{\frac12+\gamma}} 
		|x - n x_0|^b e^{-c \frac{(x-nx_0)^2}{n}} {\: \rm d} x
		&=
		\int_{-n^{\frac12+\gamma}}^{a n^{\frac12+\gamma}}
		|x|^b
		e^{-c\frac{x^2}{n}} {\: \rm d} x \\
		&=
		n^{\frac{1+b}{2}} \int_{-a n^{\gamma}}^{an^\gamma} |y|^b e^{-c y^2} {\: \rm d} y,
	\end{align*}
	and since
	\[
		\int_{-a n^{\gamma}}^{a n^\gamma} |y|^b e^{-c y^2} {\: \rm d} y =
		\int_\RR |y|^b e^{-c y^2} {\: \rm d} y + o(1),
	\]
	we get \eqref{eq:23}. The proof of \eqref{eq:24} is analogous.
\end{proof}

Let us denote by $p' \in [1, \infty]$, the conjugate exponent to $p$, that is
\[
	\frac{1}{p} + \frac{1}{p'} = 1.
\]
In the following sections we determine the asymptotic behavior of $\ell^p(V_g)$ norms of the heat kernel $k(n; o, \cdot)$
as $n$ goes to infinity. The study splits into three cases: $1 \leqslant p < 2$ (Section \ref{sec:2.1}), $p = 2$ (Section 
\ref{sec:2.2}), and $p > 2$ (Section \ref{sec:2.3}).

\subsection{The case $p \in [1, 2)$}
\label{sec:2.1}
Given $0 < \gamma < \frac{1}{6}$, for each $n \in \NN$ we set
\begin{equation}
	\label{eq:22a}
	\calN_n^p=\left\{\lambda \in P^+: \left|\frac{\lambda}{n} - \delta_p \right| \leqslant n^{-\frac{1}{2}+\gamma}\right\},
\end{equation}
and
\begin{equation}
	\label{eq:4a}
	\scrN^p_n=\big\{x \in V_g : \sigma(o, x) \in \calN_n^p\big\}.
\end{equation}
For a given $n \in \NN$, and $x \in V_g$ we often write $\delta = \frac{\sigma(o, x)}{n}$.

\begin{theorem}
	\label{thm:1}
	Let $k_n$ be a transition kernel of an admissible random walk on good vertices of an affine building $\scrX$.
	If $p \in [1, 2)$, then
	\[
		\|k_n \|_{\ell^p}
		\approx 
		n^{-\frac{r}{2p'}} 
		\varrho^{n} \kappa(s_p)^n.
	\]
	Furthermore,
	\[
		\frac{1}{\| k_n\|_{\ell^p}}
		\bigg( \sum_{x \in V_g \setminus \scrN_n^p} k(n; x)^p\bigg)^{\frac{1}{p}}
		= \calO\Big(e^{-c n^{2\gamma}} \Big).
	\]
\end{theorem}
\begin{proof}
	We start by determining the asymptotic behavior of the norm restricted to $\scrN_n^p$. We claim that
	\[
		\sum_{x \in \scrN_n^p} k(n;x)^p 
		\approx \varrho^{np}\, e^{np\psi(\delta_p)}\, n^{-\frac{r}{2}(p-1)}
	\]
	where $\psi$ is a function on $\inter \calM$ defined as
	\[
		\psi(\delta)=\langle \delta, s_p \rangle-\phi(\delta), \qquad \delta \in \inter \calM.
	\]
	In view of Lemma \ref{lem:3}\eqref{en:2:3}, for each $\alpha \in \Phi^{++}$ we have $\sprod{s_p}{\alpha} > 0$, thus
	for each $x \in \scrN_n^p$,
	\[
		\big|\sprod{\delta - \delta_p}{\alpha}\big|
		\leqslant
		|\alpha| n^{-\frac{1}{2}+\gamma},
	\]
	and so
	\[
		\sprod{\delta}{\alpha} \geqslant \sprod{\delta_p}{\alpha} - |\alpha| n^{-\frac{1}{2} + \gamma}.
	\]
	This legitimates the use of the asymptotic formula \eqref{eq:3a} in the region $\scrN_n^p$. Since
	$\bfc(s)^{-1}$ and $(\det B_{s})^{-1/2}$ are real-analytic in this region, by \eqref{eq:16}, we get
	\begin{align*}
		\sum_{x\in \scrN_n^p} k(n; x)^p
		&\approx
		n^{-\frac{pr}{2}} \varrho^{np} 
		\sum_{x \in \scrN_n^p} 
		\exp\big\{-np\phi(\delta)\big\} \chi_0(n \delta)^{-\frac{p}{2}} \big(\det B_s\big)^{-\frac{1}{2}}
		\frac{1}{\bfc(s)} \\
		&\approx
		n^{-\frac{pr}{2}} \varrho^{np} 
		\sum_{\lambda \in \calN_n^p}
		\exp\big\{-np\phi(\delta)\big\} \exp\Big\{ n \left( 1-\frac{p}{2}\right) \sprod{\delta}{\eta}\Big\} \\
		&\approx 
		n^{-\frac{pr}{2}} \varrho^{np} 
		\sum_{\lambda \in \calN_n^p}
		e^{n p \psi(\delta)}.
	\end{align*}
	Next, we compute
	\[
		D_{\delta-\delta_p} \psi(\delta_p) 
		=
		-D_{\delta-\delta_p} \phi(\delta_p) = -B^{-1}_{s_p}(\delta - \delta_p, \delta-\delta_p),
	\]
	thus there are $N \geqslant 1$ and $C > 0$ such that for all $n \geqslant N$, if $x \in \scrN_n^p$ then
	\begin{align*}
		\Big| \psi(\delta) - \psi(\delta_p) + \frac{1}{2} 
		B^{-1}_{s_p}(\delta - \delta_p, \delta-\delta_p)
		\Big|
		\leqslant
		C
		|\delta - \delta_p|^3.
	\end{align*}
	Therefore, we get
	\begin{align*}
		&
		\Big|e^{n p \psi(\delta)} - e^{n p \psi(\delta_p)} 
		e^{-p \frac{n}{2} B^{-1}_{s_p}(\delta - \delta_p, \delta-\delta_p)}
		\Big| \\
		&\qquad=
		e^{n p \psi(\delta_p) - p \frac{n}{2} B^{-1}_{s_p}(\delta - \delta_p, \delta-\delta_p)}
		\Big|1 - e^{n p \psi(\delta) - n p \psi(\delta_p) + p \frac{n}{2} B^{-1}_{s_p}(\delta - \delta_p, \delta-\delta_p)} \Big|
		\\
		&\qquad\leqslant
		C' n^{-\frac{1}{2} + 3 \gamma} e^{n p \psi(\delta_p) - p \frac{n}{2} B^{-1}_{s_p}(\delta - \delta_p, \delta-\delta_p)}.
	\end{align*}
	Consequently,
	\begin{align*}
		\sum_{\lambda \in \calN_n^p}
		e^{n p \psi(\delta)} 
		=
		e^{n p \psi(\delta_p)}
		\bigg(\sum_{\lambda \in \calN_n^p}
		e^{-p \frac{n}{2} B^{-1}_{s_p}(\delta - \delta_p, \delta-\delta_p)} \bigg)
		\Big(1 + \calO\big(n^{-\frac{1}{2}+3\gamma}\big)\Big)
	\end{align*}
	which by \eqref{eq:23} leads to
	\begin{equation}
		\label{eq:29}
		\sum_{x \in \scrN_n^p} k(n; x)^p
		\approx
		n^{-\frac{r}{2}(p-1)} \varrho^{np}
		e^{n\psi(\delta_p)}.
	\end{equation}
	Let us next observe that $\nabla \psi(\delta_p) = 0$. Since $\phi$ is strictly convex in $\inter \calM$, $\psi$ is 
	strictly concave. Hence, there is $c > 0$ such that
	\begin{equation}
		\label{eq:26}
		\psi(\delta)-\psi(\delta_p) \leqslant -c |\delta-\delta_p|^2
	\end{equation}
	for all $\delta \in \calM$. Consider now $x \in V_g \setminus \scrN_n^p$. By Lemma \ref{lem:4} and \eqref{eq:26}, we get
	\begin{align*}
		\sum_{x \in V_g\setminus \scrN_n^p} k(n; x)^p
		&\lesssim
		\varrho^{n p} e^{n p \psi(\delta_p)} 
		\sum_{\lambda \in P^+ \setminus \calN_n^p}
		e^{n p (\psi(\delta) - \psi(\delta_p))} \\
		&\lesssim
		\varrho^{n p} e^{n p \psi(\delta_p)} 
		\sum_{\lambda \in P^+ \setminus \calN_n^p}
		e^{-c n p|\delta-\delta_p|^2} \\
		&\lesssim \varrho^{n p} e^{np\psi(\delta_p)} e^{-\frac{1}{2} cn^{2\gamma} p }
		\sum_{\lambda \in P} e^{-\frac{1}{2} c n p |\delta - \delta_p|^2}.
	\end{align*}
	In view of \eqref{eq:23}, we get
	\[
		\sum_{x \in V_g \setminus \scrN_n^p} k(n; x)^p
		\lesssim
		n^{\frac{r}{2}} \varrho^{n p} e^{np\psi(\delta_p)} e^{-c'' n^{2\gamma}}
	\]
	which together with \eqref{eq:29} completes the proof.
\end{proof}

\subsection{The case $p = 2$}
\label{sec:2.2}
Let
\begin{equation}
	\label{eq:38}
	0 < \gamma < \frac{1}{4|\Phi^{++}|}
	\quad\text{and}\quad
	\gamma' = 2 \gamma |\Phi^{++}|.
\end{equation} 
For each $n \in \NN$, we set
\[
	\calN_n^2 =
	\left\{\lambda \in P^+ : 
	n^{\frac{1}{2}-\gamma} \leqslant | \lambda | \leqslant n^{\frac{1}{2}+\gamma}, \text{ and }
	\sprod{\lambda}{\alpha} \geqslant n^{\frac{1}{2} - \gamma'}
	\text{ for all } \alpha\in \Phi^{+}\right\},
\]
and
\begin{equation}
	\label{eq:4b}
	\scrN_n^2 = \big\{ x \in V_g : \sigma(o, x) \in \calN_n^2 \big\}.
\end{equation}
\begin{theorem}
	\label{thm:2}
	Let $k_n$ be a transition kernel of an admissible random walk on good vertices of an affine building $\scrX$.
	Then
	\[
		\|k_n \|_{\ell^2}
		\approx 
		n^{-\frac{r}{4}-\frac{|\Phi^{++}|}{2}} \varrho^{n}.
	\]
	Furthermore,
	\[
		\frac{1}{\| k_n\|_{\ell^2}}
		\bigg( \sum_{x \in V_g \setminus \scrN_n^2} k(n; x)^2 \bigg)^\frac{1}{2} = 
		\calO\Big(n^{-\gamma|\Phi^{++}|}\Big).
	\]
\end{theorem}
\begin{proof} 
	We first determine the asymptotic behavior of the norm restricted to $\calN_n^2$. For $v \in \mathfrak{a}$ we set
	\[
		m(v) = \min \big\{\sprod{v}{\alpha} : \alpha  \in \Phi^{++}\big\}.
	\]
	For $n \in \NN$, let
	\begin{align*}
		\calA_n &= \Big\{\lambda \in P^+ : n^{\frac{1}{2}-\gamma} \leqslant |\lambda| \leqslant n^{\frac12+\gamma} \Big\},
		&\calB_n &= \Big\{\lambda \in P^+: |\lambda| \leqslant n^{\frac12+\gamma} \Big\}, \\
		\scrA_n &= \{x \in V_g : \sigma(o, x) \in \calA_n\}, 
		&\scrB_n &= \{x \in V_g : \sigma(o, x) \in \calB_n\}. 
	\end{align*}
	Applying the asymptotic formula \eqref{eq:3b} together with Proposition \ref{prop:1}, we get
	\begin{align*}
		\sum_{x \in \scrN_n^2} k(n; x)^2 
		&\approx
		\varrho^{2n} n^{-r - 2 |\Phi^{++}|} 
		\sum_{x \in \scrN_n^2} e^{- 2 n \phi(\delta)} \chi_0(\sigma(o, x))^{-1} 
		\prod_{\alpha \in \Phi^{++}} (1 + \sprod{\alpha}{\sigma(o, x)}^2)  \\
		&\approx
		\varrho^{2n} n^{-r - 2 |\Phi^{++}|} 
		\sum_{\lambda \in \calN_n^2} 
		e^{-2 n \phi(\delta)} \prod_{\alpha \in \Phi^{++}} (1 + \sprod{\alpha}{\lambda}^2).
	\end{align*}
	Now, in view of \eqref{eq:33b}, there are $c' \geqslant c > 0$ such that 
	\begin{align*}
		\sum_{\lambda \in \calN_n^2} 
		e^{-2 n \phi(\delta)} \prod_{\alpha \in \Phi^{++}} 
		(1 + \sprod{\alpha}{\lambda}^2)
		\leqslant
		\sum_{\lambda \in \calB_n} e^{-c \frac{|\lambda|^2}{n}}
		\prod_{\alpha \in \Phi^{++}}
		(1 + \sprod{\alpha}{\lambda}^2),
	\end{align*}
	and
	\begin{align*}
		\sum_{\lambda \in \calN_n^2} 
		e^{-2 n \phi(\delta)} \prod_{\alpha \in \Phi^{++}} (1 + \sprod{\alpha}{\lambda}^2)
		&\geqslant
		\sum_{\lambda \in \calA_n : m(\lambda) > 1 } e^{-c' \frac{|\lambda|^2}{n}}
		\prod_{\alpha \in \Phi^{++}}
		(1 + \sprod{\alpha}{\lambda}^2).
	\end{align*}
	Observe that if $\lambda \in \calA_n$ and $m(\lambda) > 1$, then $\lambda = \lambda' + \rho$ and 
	\[
		\frac{1}{2} n^{\frac12-\gamma} \leqslant |\lambda'| \leqslant 2 n^{\frac12+\gamma}.
	\]
	Therefore, by Lemma \ref{lem:1}, we get
	\[
		\sum_{\lambda \in \calN_n^2} 
		e^{-2 n \phi(\delta)} \prod_{\alpha \in \Phi^{++}} (1 + \sprod{\alpha}{\lambda})^2
		\approx
		n^{\frac{r}{2} + |\Phi^{++}|},
	\]
	and consequently,
	\begin{equation}
		\label{eq:34}
		\sum_{x \in \scrN_n^2} k(n; x)^2 \approx \varrho^{2n} n^{-\frac{r}{2} - |\Phi^{++}|}.
	\end{equation}
	Next, let us consider the region $V_g \setminus \scrB_n$. By Lemma \ref{lem:4} and \eqref{eq:33b}, there is $c > 0$ 
	such that 
	\begin{align*}
		\sum_{x \in V_g \setminus \scrB_n}
		k(n; x)^2 
		&\lesssim 
		\varrho^{2n} 
		\sum_{\lambda \in P^+ \setminus \calB_n} e^{-c \frac{|\lambda|^2}{n}} \\
		&\lesssim
		\varrho^{2n}
		e^{-\frac{c}{2} n^{2\gamma}} \sum_{\lambda \in P} e^{-c \frac{|\lambda|^2}{2n}},
	\end{align*}
	thus
	\begin{equation}
		\label{eq:36}
		\sum_{x \in V_g \setminus \scrB_n} k(n; x)^2
		\lesssim \varrho^{2n} n^{-\frac{r}{2} - |\Phi^{++}|} e^{-\frac{c}{4} n^{2\gamma}}.
	\end{equation}
	Let us observe that in the region $x \in \scrB_n \setminus \scrA_n$, for each $\alpha \in \Phi^{++}$, 
	we have
	\[
		1 + \sprod{\lambda}{\alpha}^2 \lesssim n^{1- 2 \gamma}, 
	\]
	hence by the asymptotic formula \eqref{eq:3b} together with Proposition \ref{prop:1} and \eqref{eq:33b}, we get
	\begin{align*}
		\sum_{x \in \scrB_n \setminus \scrA_n} k(n; x)^2
		&\lesssim
		\varrho^{2n} n^{-r - 2 |\Phi^{++}|} 
		\sum_{x \in \scrB_n \setminus \scrA_n}
		e^{-2n \phi(\delta)} \chi_0(\sigma(o, x))^{-1} 
		\prod_{\alpha \in \Phi^{++}} \big(1 + \sprod{\sigma(o, x)}{\alpha}^2\big) \\
		&\lesssim
		\varrho^{2n} n^{-r -|\Phi^{++}|-2 \gamma|\Phi^{++}|}
		\sum_{\lambda \in \calB_n \setminus \calA_n} e^{-c \frac{|\lambda|^2}{n}}.
	\end{align*}
	Therefore, we obtain
	\begin{equation}
		\label{eq:37}
		\sum_{x \in \scrB_n \setminus \scrA_n} k(n; x)^2
		\lesssim
		\varrho^{2n} n^{-\frac{r}{2} -|\Phi^{++}|} n^{-2 \gamma|\Phi^{++}|}.
	\end{equation}
	Lastly, we need to consider the region $\scrA_n \setminus \scrN_n^2$. For $x \in \scrA_n \setminus \scrN_n^2$, there
	is $\beta \in \Phi^{++}$ such that
	\[
		\sprod{\sigma(o, x)}{\beta} \leqslant n^{\frac{1}{2}-\gamma'}.
	\]
	However, for all $\alpha \neq \beta$, we have
	\[
		\sprod{\sigma(o, x)}{\alpha} \leqslant C n^{\frac12+\gamma},
	\]
	thus by \eqref{eq:38}
	\[
		\prod_{\alpha \in \Phi^{++}} \big(1 + \sprod{\sigma(o, x)}{\alpha}^2\big)
		\lesssim 
		n^{|\Phi^{++}|} n^{-2(\gamma' - \gamma (|\Phi^{++}| - 1))}
		\leqslant
		n^{|\Phi^{++}|} n^{-2\gamma (|\Phi^{++}| + 1)}.
	\]
	Therefore, by the asymptotic formula \eqref{eq:3b} together with Proposition \ref{prop:1} and \eqref{eq:33b}, we obtain
	\begin{align}
		\nonumber
		\sum_{x \in \scrA_n \setminus \scrN_n^2}
		k(n; x)^2
		&\lesssim
		\varrho^{2n} n^{-r - 2 |\Phi^{++}|} 
		\sum_{\lambda \in \calA_n \setminus \calN_n^2}
		e^{-2n \phi(\delta)}
		\prod_{\alpha \in \Phi^{++}} \big(1 + \sprod{\lambda}{\alpha}^2\big) \\
		\label{eq:39}
		&\lesssim
		\varrho^{2n} n^{-\frac{r}2 - |\Phi^{++}|}  n^{-2\gamma (|\Phi^{++}| + 1)}.
	\end{align}
	Using estimates \eqref{eq:36}, \eqref{eq:37} and \eqref{eq:39} together with the asymptotic \eqref{eq:34} 
	we complete the proof.
\end{proof}
	
\subsection{The case $p \in (2, \infty]$}
\label{sec:2.3}
Let $p \in (2, \infty]$. Fix a sequence of positive numbers $(r_n : n \in \NN)$ such that
\begin{equation}
	\label{eq:40}
	\lim_{n \to \infty} \frac{r_n}{\log n} = \infty, \quad\text{ and }\quad \lim_{n \to \infty} \frac{r_n}{\sqrt{n}} = 0.
\end{equation}
Then for each $n \in \NN$, we set
\[
	\calN_n^p = \big\{\lambda \in P^+ : |\lambda| \leqslant r_n \big\},
\]
and
\begin{equation}
	\label{eq:4c}
	\scrN_n^p= \{x \in V_g : \sigma(o, x) \in \calN_n^p\}.
\end{equation}

\begin{theorem}
	\label{thm:3}
	Let $k_n$ be a transition kernel of an admissible random walk on good vertices of an affine building $\scrX$.
	Then
	\[
		\|k_n\|_{\ell^p} \approx 
		n^{-\frac{r}{2}-|\Phi^{++}|} \varrho^{n}.
	\]
	Furthermore,
	\[
		\lim_{n \to \infty}
		\frac{1}{\| k_n \|_{\ell^p}}
		\bigg(
		\sum_{x\notin \mathcal{N}_n^p} k(n; x)^p\bigg)^\frac{1}{p} = 
		\calO\Big(e^{-c_p r_n}\Big)
	\]
	for certain $c_p > 0$.
\end{theorem}

\begin{proof}
	We consider only the case where $p \in (2, \infty)$, as the case $p = \infty$ can be proven in a similar manner. 
	In view of \eqref{eq:40}, in the region $\scrN_n^p$, we can apply the asymptotic formula \eqref{eq:3b}
	together with Proposition \ref{prop:1} to get
	\begin{align*}
		\sum_{x \in \scrN_n^p} k(n; x)^p
		\approx
		\varrho^{np} n^{-\frac{pr}2 - p|\Phi^{++}|} 
		\sum_{\lambda \in \calN_n^p} e^{-n p \phi(\delta)} \chi_0(\lambda)^{1-\frac{p}2} 
		\prod_{\alpha \in \Phi^{++}} \big(1 + \sprod{\lambda}{\alpha}^2\big).
	\end{align*}
	By \eqref{eq:33b},
	\[
		0 \leqslant n \phi(\delta) \lesssim n |\delta|^2 \leqslant \Big(\frac{r_n}{\sqrt{n}}\Big)^2,
	\]
	thus by \eqref{eq:40},
	\begin{align*}
		\sum_{\lambda \in \calN_n^p} e^{-n p \phi(\delta)} \chi_0(\lambda)^{1-\frac{p}2} \prod_{\alpha \in \Phi^{++}}
		\big(1 + \sprod{\lambda}{\alpha}^2\big)
		\approx
		\sum_{\lambda \in \calN_n^p} e^{-(p-2)\sprod{\eta}{\lambda}} \prod_{\alpha \in \Phi^{++}}
		\big(1 + \sprod{\lambda}{\alpha}^2\big) \approx 1.
	\end{align*}
	Therefore,
	\[
		\sum_{x \in \scrN_n^p} k(n; x)^p \approx \varrho^{np} n^{-\frac{pr}2 - p|\Phi^{++}|}.
	\]
	It remains to consider the region $V_g \setminus \scrN_n^p$. In view of \eqref{eq:3} and \eqref{eq:3b} we have
	\begin{align*}
		\sum_{x \in V_g \setminus \scrN_n^p} k(n; x)^p 
		&\lesssim 
		\varrho^{np} 
		\sum_{\lambda \in P^+ \setminus \calN_n^p} 
		\chi_0(\lambda)^{1-\frac{p}2} e^{-c \frac{|\lambda|^2}{n}} \\
		&\leqslant
		\varrho^{np}
		\sum_{\lambda \in P^+ \setminus \calN_n^p}
		e^{-(p-2)\sprod{\eta}{\lambda}}.
	\end{align*}
	Setting
	\[
		c_p = \frac{p-2}{4} \min\bigg\{\frac{1}{|\lambda|} \sprod{\eta}{\lambda} : \lambda \in P^+ \bigg\} > 0
	\]
	we obtain
	\begin{align*}
		\sum_{x \in V_g \setminus \scrN_n^p} k(n; x)^p
		&\lesssim
		\varrho^{np} \sum_{\lambda \in P^+ \setminus \calN_n^p} e^{-4 c_p |\lambda|} \\
		&\lesssim
		\varrho^{np} e^{-2 c_p r_n}.
	\end{align*}
	Since
	\[
		c_p r_n  
		= c_p \frac{r_n}{\log n} \log n
		\geqslant p \bigg(\frac{r}{2} + |\Phi^{++}| + 1\bigg) \log n,
	\]
	the theorem follows.
\end{proof}

\section{Ratio limits for heat kernels}
\label{sec:3}
In this section we present some technical results, concerning the asymptotic behavior of quotients of heat kernel on
certain regions. The results of this section will be crucial to the description of the asymptotic behavior of caloric functions 
on affine buildings. 
\begin{theorem}
	\label{thm:4}
	Let $k_n$ be a transition kernel of an admissible random walk on good vertices of an affine building $\scrX$.
	For $\xi > 0$, $\varepsilon > 0$ and $y \in V_g$, and any sequence $(x_n : n \in \NN)$ of good vertices such that
	$\delta_n = \tfrac{1}{n} \sigma(o, x_n)$ satisfies
	\begin{equation}
		\label{eq:46}
		\sprod{\delta_n}{\alpha} 
		\geqslant 
		\xi \quad\text{for all } \alpha\in \Phi^{+}, \\
	\end{equation}
	and
	\[
		\dist(\delta_n, \partial \mathcal{M}) \geqslant \varepsilon,
	\]
	we have
	\[ 
		\frac{k(n; y, x_n)}{k(n; o, x_n)}
		=\chi_0^{-\frac{1}{2}}(\sigma(y,x_n)-\sigma(o,x_n))
		e^{-\sprod{s_n}{\sigma(y,x_n)-\sigma(o,x_n)}} \big(1+\calO(n^{-1})\big)
	\]
	where $s_n=\nabla\phi(\delta_n)$. The implicit constant in the error term depends only on $\xi$, $\varepsilon$, $A$, 
	and $y$.
\end{theorem}
\begin{proof}
	The hypothesis of the theorem entitles us to use the asymptotic formula \eqref{eq:3a}. We get
	\begin{equation}
		\label{eq:43}
		k(n; x_n) = n^{-\frac{r}2} \varrho^n e^{-n\phi(\delta)} \chi_0(\sigma(o,x_n))^{-\frac{1}{2}} 
		(\det B_{s_n})^{-\frac{1}{2}} \frac{1}{\bfc(s_n)} \Big(C_0 + \mathcal{O}(n^{-1})\Big).
	\end{equation}
	Let $\tilde{\delta}_n = \frac{1}{n} \sigma(y, x_n)$. Then by \eqref{eq:14}
	\begin{equation}
		\label{eq:44}
		|\tilde{\delta}_n - \delta_n| \leqslant \frac{|\sigma(o, y)|}{n}.
	\end{equation}
	There is $N \geqslant 1$ such that for all $n \geqslant N$,
	\[
		\frac{|\sigma(o, y)|}{n} \leqslant \tfrac{1}{2} \min\{\xi, \varepsilon\}.
	\]
	Then for all $n \geqslant N$, and $\alpha \in \Phi^+$,
	\[
		\sprod{\tilde{\delta}_n}{\alpha} \geqslant \tfrac{1}{2} \xi, \qquad \text{and}\qquad
		\dist(\tilde{\delta}, \partial \calM) \geqslant \tfrac{1}{2} \varepsilon,
	\]
	which once again legitimizes the use of asymptotic formula \eqref{eq:3a}. We get
	\begin{equation}
		\label{eq:49}
		k(n; y, x_n) = n^{-\frac{r}2} \varrho^n e^{-n\phi(\tilde{\delta})} \chi_0(\sigma(y, x_n))^{-\frac{1}{2}} 
		(\det B_{\tilde{s}_n})^{-\frac{1}{2}} \frac{1}{\bfc(\tilde{s}_n)} \Big(C_0 + \mathcal{O}(n^{-1})\Big)
	\end{equation}
	where $\tilde{s}_n = s(\tilde{\delta}_n)$. Next, let us recall that $s$ is real-analytic on every compact subset of
	$\inter \calM$, thus by \eqref{eq:44},
	\begin{equation}
		\label{eq:48}
		|\tilde{s}_n - s_n| 
		\leqslant C |\tilde{\delta}_n - \delta_n| \leqslant C n^{-1}.
	\end{equation}
	Analogously, one can show that
	\[
		|B_{\tilde{s}_n} - B_{s_n}| \leqslant C'n^{-1}.
	\]
	Furthermore, there is $\zeta > 0$ such that
	\begin{equation}
		\label{eq:47}
		\sprod{s_n}{\alpha\spcheck} > \zeta.
	\end{equation}
	Indeed, otherwise, we can select a subsequence $(\delta_{n_k} : k \in \NN)$ such that 
	\begin{equation}
		\label{eq:45}
		\sprod{s_{n_k}}{\alpha\spcheck} \leqslant \frac{1}{k}.
	\end{equation}
	Since $(\delta_n)$ is contained in the compact set $\{\delta \in \calM : \dist(\delta, \calM) \geqslant \xi\}$,
	we can assume that the sequence $(\delta_{n_k} : k \in \NN)$ converges to $\delta_0$. By \eqref{eq:45}, 
	$\sprod{s(\delta_0)}{\alpha\spcheck} = 0$, which implies that $\sprod{\delta_0}{\alpha} = 0$,
	contradicting \eqref{eq:46}.

	In view of \eqref{eq:47} and \eqref{eq:48}, we get
	\[
		\sprod{\tilde{s}_n}{\alpha\spcheck} = \sprod{s_n}{\alpha\spcheck}
		(1+\calO(n^{-1})),
	\]
	which leads to
	\[
		\bfc(\tilde{s}_n) = \bfc(s_n) (1+\calO(n^{-1})).
	\]
	Consequently, by \eqref{eq:49},
	\[
		k(n; y, x_n) = 
		n^{-\frac{r}2} \varrho^n e^{-n\phi(\tilde{\delta})} \chi_0(\sigma(y, x_n))^{-\frac{1}{2}}
		(\det B_{s_n})^{-\frac{1}{2}} \frac{1}{\bfc(s_n)} \Big(C_0 + \mathcal{O}(n^{-1})\Big).
	\]
	Hence, by \eqref{eq:43},
	\begin{equation*}
		\frac{k(n; y, x_n)}{k(n;o,x_n)} = 
		\chi_0^{-\frac{1}{2}}(\sigma(y,x_n)-\sigma(o,x_n))\,e^{-n(\phi(\tilde{\delta}_n)-\phi(\delta_n))}
		(1+\calO(n^{-1})).
	\end{equation*}
	In order to treat the terms in the exponential, let us notice that
	\begin{align*}
		\phi(\tilde{\delta}_n)
		&=\phi(\delta_n)+\sprod{\nabla\phi(\delta_n)}{\tilde{\delta}_n-\delta_n}+\calO(|\tilde{\delta}_n-\delta_n|^2)\\
		&=\phi(\delta_n)+\langle s_n,\widetilde{\delta}_n-\delta_n\rangle+\mathcal{O}(n^{-2}),
	\end{align*}
	which completes the proof of the tentative asymptotic formula.
\end{proof}	

As an immediate corollary, owing to the fact that $\sigma(o,x_n)-\sigma(y,x_n)$ is bounded, we get the following result.
\begin{corollary}
	\label{cor:2}
	Let $p \in [1, 2)$. Under the hypothesis of Theorem \ref{thm:4}, if $(\delta_n : n \in \NN)$ converges to 
	$\delta_p \in \calM$, then 
	\[ 
		\frac{k(n; y, x_n)}{k(n;o,x_n)}
		=
		\chi_0^{-\frac{1}{p}}(\sigma(y,x_n)-\sigma(o,x_n))\, 
		(1+\calO(|\delta_n - \delta_p|) + \calO(n^{-1})).
	\]
\end{corollary}

The next result concerns ratio limits near the origin. Let us recall the definition of spherical functions \eqref{eq:5}.
\begin{theorem}
	\label{thm:5}
	Let $k_n$ be a transition kernel of an admissible random walk on good vertices of an affine building $\scrX$.
	If $(x_n)$ is a sequence of good vertices such that
	\begin{equation}
		\label{eq:78}
		\lim_{n \to \infty} \frac{\sigma(o, x_n)}{n} = 0,
	\end{equation}
	then for each $y \in V_g$,
	\[
		\frac{k(n; y, x_n)}{k(n; o, x_n)}
		=
		\frac{\mathbf{\Phi}(\sigma(y, x_n))}{\mathbf{\Phi}(\sigma(o, x_n))}  
		\left(1+\calO\left(\frac{\sigma(o, x_n)}{n} \right)\right).
	\]
\end{theorem} 
\begin{proof}
	Let
	\[
		\delta_n = \frac{1}{n} \sigma(o, x_n), \qquad\text{and}\qquad
		\tilde{\delta}_n = \frac{1}{n} \sigma(y, x_n),
	\]
	and $s_n = s(\delta_n)$, $\tilde{s}_n = s(\tilde{\delta}_n)$. By \eqref{eq:12} and \eqref{eq:78}
	\[
		\lim_{n \to \infty} \tilde{\delta}_n = 0,
	\]
	which allows us to use the asymptotic formula \eqref{eq:3b}. We obtain
	\[
		\frac{k(n; y, x_n)}{k(n; o, x_n)}
		=
		\frac{\mathbf{\Phi}(\sigma(y, x_n))}{\mathbf{\Phi}(\sigma(o, x_n))}
		e^{-n\phi(\delta_n) + n \phi(\tilde{\delta}_n)} 
		\Big(1 + \calO(|\delta_n|)\Big).
	\]
	To finish the proof, we need to show that
	\[
		n \big(\phi(\delta_n) - \phi(\tilde{\delta}_n)\big) = \calO(|\delta_n|).
	\]
	By \eqref{eq:20}, we have
	\[
		\sprod{\tilde{\delta}_n}{s_n} - \log \kappa(s_n)
		\leqslant 
		\sprod{\tilde{\delta}_n}{\tilde{s}_n} 
		- 
		\log \kappa(\tilde{s}_n)= \phi(\tilde{\delta}_n).
	\]
	Hence, by \eqref{eq:14}
	\begin{align*}
		\phi(\delta_n)- \phi(\tilde{\delta}_n) 
		&\leqslant 
		\phi(\delta_n)- \sprod{\tilde{\delta}_n}{s_n} + \log\kappa(s_n)\\
		&=\frac{1}{n} \sprod{(\sigma(o, x_n) - \sigma(y, x_n))}{s_n} = \mathcal{O}(|s_n|/n).
	\end{align*}
	Since $s$ is real-analytic on any compact subset of $\inter \calM$, by the mean value theorem
	\[
		|s_n| = |s(\delta_n) - s(0)| \leqslant C |\delta_n|,
	\]
	and the claim follows.
\end{proof}

The above result can be further elaborated away from the walls.
	
\begin{corollary}
	\label{cor:1}
	Under the hypothesis of Theorem \ref{thm:5}, if $(x_n)$ is a sequence of good vertices such that
	for each $\alpha \in \Phi^+$,
	\[
		\lim_{n \to \infty} \sprod{\sigma(o, x_n)}{\alpha} = \infty,
	\]
	then
	\[
		\frac{k(n; y, x_n)}{k(n; o, x_n)}=
		\chi_0^{-\frac{1}{2}}(\sigma(y,x_n)-\sigma(o,x_n)) + o(1).
	\]
\end{corollary}
\begin{proof}
	In view of \eqref{eq:11}, for each $\alpha \in \Phi^+$,
	\[
		\lim_{n \to \infty} \sprod{\sigma(y, x_n)}{\alpha} = \infty.
	\]
	By Proposition \ref{prop:1}, we have
	\[
		\frac{\mathbf{\Phi}(\sigma(y, x_n))}{\mathbf{\Phi}(\sigma(o, x_n))}
		=
		\chi_{0}^{-\frac{1}{2}}(\sigma(y,x_n)-\sigma(o,x_n))
		\bigg(\prod_{\alpha \in \Phi^+}
		\frac{1 + \sprod{\sigma(y, x_n)}{\alpha}}{1+ \sprod{\sigma(o, x_n)}{\alpha}}\bigg)(1+o(1)).
	\]
	By \eqref{eq:14}, for each $\alpha \in \Phi^+$,
	\[
		\frac{1 + \sprod{\sigma(y, x_n)}{\alpha}}{1+ \sprod{\sigma(o, x_n)}{\alpha}} - 1
		=
		\frac{\sprod{\sigma(y, x_n) - \sigma(o, x_n)}{\alpha}}{1+ \sprod{\sigma(o, x_n)}{\alpha}}
		=
		\calO(\sprod{\sigma(o, x_n)}{\alpha}^{-1}),
	\]
	thus to complete the proof it is enough to invoke Theorem \ref{thm:5}.
\end{proof}

\section{Asymptotic behavior of caloric functions}
\label{sec:4}
In this section we discuss the $\ell^p$ asymptotic behavior of caloric functions on affine buildings. We first introduce 
appropriate mass functions. We next prove our main result, which concerns the $\ell^p$ asymptotic behavior of caloric 
functions coming from compactly supported initial data. Lastly, we extend this result to further classes of initial data.

\subsection{Helgason transform}
Given a finitely supported function $f$ on good vertices $V_g$, its Helgason transform is defined by the formula
\[
	\calH f(z, \omega) = \sum_{y \in V_g} f(y) \chi_0(h(o, y; \omega))^{\frac{1}{2}} e^{-\sprod{z}{h(o, y; \omega)}},
	\qquad (z, \omega) \in \mathfrak{a}_\CC \times \Omega.
\]
In the following propositions we consider radial functions on $V_g$. Let us recall that $f$ is called \emph{radial} if
for all good vertices $x$ and $y$, if $\sigma(o, x)=\sigma(o, y)$, then $f(x)=f(y)$. If $f$ is radial we often identify it
with its profile $\tilde{f}$, where $\tilde{f}(\lambda) = f(x)$, for any $x \in V_\lambda(o)$.
\begin{proposition}
	\label{prop:4}
	For all $\lambda \in P^+$, $z \in \mathfrak{a}_\CC$ and $\omega \in \Omega$, and $x \in V_g$,
	\begin{equation}
		\label{eq:61}
		\sum_{y \in V_\lambda(x)} \chi_0(h(x, y; \omega))^{\frac12} e^{-\sprod{z}{h(x, y; \omega)}}
		=
		N_\lambda P_\lambda(-z).
	\end{equation}
\end{proposition}
\begin{proof}
	Let us observe that for each $\lambda \in P^+$ and $z \in \mathfrak{a}_\CC$, the sum
	\[
		\sum_{y \in V_\lambda(x)} \chi_0(h(x, y; \omega))^{\frac12} e^{-\sprod{z}{h(x, y; \omega)}}
	\]
	is independent of $\omega \in \Omega$. Indeed, we have
	\begin{equation}
		\label{eq:59}
		\sum_{y \in V_\lambda(x)} \chi_0(h(x, y; \omega))^{\frac12} e^{-\sprod{z}{h(x, y; \omega)}}
		=
		\sum_{\mu \in \Pi_\lambda} e^{\sprod{\eta-z}{\mu}} \#\{y \in V_\lambda(x) : h(x, y; \omega) = \mu\}.
	\end{equation}
	In view of \cite[Lemma 3.19]{park2}, the set
	\[
		\#\{y\in V_{\lambda}(x): h(x,y;\omega)=\mu\}
	\]
	is independent of $\omega \in \Omega$, thus the left-hand side of \eqref{eq:59}. Consequently,
	\begin{align*}
		\sum_{y \in V_\lambda(x)} \chi_0(h(x, y; \omega))^{\frac12} e^{-\sprod{z}{h(x, y; \omega)}}
		&=
		\int_\Omega
		\sum_{y \in V_\lambda(x)} \chi_0(h(x, y; \omega))^{\frac12} e^{-\sprod{z}{h(x, y; \omega)}}
		\: \nu_x({\rm d} \omega) \\
		&=
		\sum_{y \in V_\lambda(x)}
		\int_\Omega
		\chi_0 (h(x, y; \omega))^{\frac12} e^{-\sprod{z}{h(x, y; \omega)}}
		\: \nu_x({\rm d} \omega),
	\end{align*}
	which in view of the integral representation \eqref{eq:9} is equal to $N_\lambda P_\lambda(-z)$,
	proving \eqref{eq:61}.
\end{proof}
Thanks to Proposition \ref{prop:1}, the Helgason transform can be defined for radial functions in $\ell^1(V_g)$. Indeed,
if $z \in [-1, 1]\eta + i U_0$, and $\omega \in \Omega$,
\begin{align*}
	\sum_{y \in V_g} \Big| f(y) \chi_0(h(o, y; \omega))^{\frac12} e^{-\sprod{z}{h(o, y; \omega)}} \Big|
	&\leqslant
	\sum_{y \in V_g} |f(y)| \chi_0(h(o, y; \omega))^{\frac12} e^{-\sprod{\Re z}{h(o, y; \omega)}} \\
	&\leqslant
	\sum_{\lambda \in P^+} |f(\lambda)| \sum_{y \in V_\lambda(o)} \chi_0(h(o, y; \omega))^{\frac12} 
	e^{-\sprod{\Re z}{h(o, y; \omega)}} \\
	&=
	\sum_{\lambda \in P^+} |f(\lambda)| N_\lambda P_{\lambda}(-\Re z) \\
	&\leqslant
	\sum_{\lambda \in P^+} |f(\lambda)| N_\lambda = \|f\|_{\ell^1}
\end{align*}
where the last inequality follows by Lemma \ref{lem:2}. Consequently, if $f$ is radial and in $\ell^1(V_g)$
then for all $z \in [-1, 1]\eta + iU_0$ and $\omega \in \Omega$,
\begin{equation}
	\label{eq:74}
	|\calH f(z, \omega)| \leqslant \|f\|_{\ell^1}.
\end{equation}
Moreover, we have the following corollary.
\begin{corollary}
	\label{cor:3}
	Let $f$ be a radial function belonging to $\ell^1(V_g)$. If $z \in [-1, 1]\eta + iU_0$ and $\omega \in \Omega$, then
	\[
		\calH f(z, \omega)
		=
		\sum_{y \in V_g} f(y) \vphi_{-z}(y).
	\]
\end{corollary}

In view of Corollary \ref{cor:3}, if $f$ is a radial function from $\ell^1(V_g)$ and $z \in [-1, 1]\eta + iU_0$, then 
the Helgason transform $\calH f(z, \cdot)$ is constant on $\Omega$, thus we often just write $\calH f(z)$.

Next, let us observe that for any two good vertices $x, y \in V_g$, by \eqref{eq:9} and \eqref{eq:52}, we have
\begin{align}
	\nonumber
	P_{\sigma(x, y)}(z)
	&= \int_\Omega \chi_0\big(h(x, y; \omega)\big)^{\frac12} e^{\sprod{z}{h(x, y; \omega)}} \nu_x({\rm d} \omega) \\
	\nonumber
	&= \int_\Omega \chi_0\big(h(o, x; \omega)\big)^{-\frac12} e^{-\sprod{z}{h(o, x; \omega)}} 
	\chi_0\big(h(o, y; \omega)\big)^{\frac12} e^{\sprod{z}{h(o, y; \omega)}} \, \nu_x({\rm d} \omega) \\
	\label{eq:2}
	&=
	\int_\Omega \chi_0\big(h(o, x; \omega)\big)^{\frac12} e^{-\sprod{z}{h(o, x; \omega)}}
    \chi_0\big(h(o, y; \omega)\big)^{\frac12} e^{\sprod{z}{h(o, y; \omega)}} \, \nu({\rm d} \omega)
\end{align}
where in the last equality we have also used \eqref{eq:15}.

Let $f$ be a radial function in $\ell^p(V_g)$, $p \geqslant 1$. Then for each $g \in \ell^{p'}(V_g)$, we can 
define the convolution $f \times g$ by the formula
\[
	f \times g(x) = \sum_{y \in V_g} f(y) g(\sigma(y, x)), \qquad x \in V_g.
\]
Let us observe that if both $f$ and $g$ are radial then $f \times g = g \times f$. Indeed, setting 
$\lambda^\star = -w_0. \lambda$, we can write
\begin{align*}
	f \times g(x) 
	&= \sum_{\lambda \in P^+} \sum_{\mu \in P^+} f(\lambda) g(\mu) \#\big(V_{\mu^\star}(x) \cap V_\lambda(o)\big)
\intertext{and since the building is strongly vertex regular, \cite[Theorem 5.21]{park3}}
	&= \sum_{\lambda \in P^+} \sum_{\mu \in P^+} f(\lambda) g(\mu) \#\big(V_{\mu}(o) \cap V_{\lambda^\star}(x)\big) \\
	&= g \times f(x).
\end{align*}
Let us recall the definition of the Plancherel measure \eqref{eq:80} and \eqref{eq:81}.
\begin{theorem}
	\label{thm:12}
	If $f$ and $g$ are finitely supported complex-valued functions on $V_g$, then 
	\[
		\sum_{y \in V_g} f(y) \overline{g(y)} =
		\int_{\scrD} \int_\Omega \calH f(i z, \omega) \calH \overline{g} (-i z, \omega)
		\nu({\rm d} \omega) \pi({\rm d} z).
	\]
\end{theorem}
\begin{proof}
	Given $z \in \scrD$ and $\omega \in \Omega$, we can write
	\begin{align*}
		&\calH f(i z, \omega) \calH \overline{g}( -i z, \omega) \\
		&\qquad\qquad=
		\sum_{y, y' \in V_g} f(y) \overline{g(y')} 
		\chi_0(h(o, y; \omega))^{\frac12} \chi_0(h(o, y'; \omega))^{\frac12} e^{-\sprod{i z}{h(o, y; \omega)}}
		e^{\sprod{i z}{h(o, y'; \omega)}} \\
		&\qquad\qquad=
		\sum_{y, y' \in V_g} f(y) \overline{g(y')}
		\chi_0(h(o, y; \omega)) \chi_0(h(y, y'; \omega))^{\frac12}
		e^{\sprod{iz}{h(y, y'; \omega)}} 
	\end{align*}
	where we have used the horocycle relation \eqref{eq:52}. Hence, by \eqref{eq:15} and \eqref{eq:9}, we get
	\begin{align*}
		&\int_\Omega \calH f(iz, \omega) \calH \overline{g}( -i z, \omega) \nu({\rm d} \omega) \\
		&\qquad\qquad=
		\sum_{y, y' \in V_g} f(y) \overline{g(y')}
		\int_\Omega \chi_0(h(o, y; \omega)) \chi_0(h(y, y'; \omega))^{\frac12} e^{\sprod{i z}{h(y, y'; \omega)}}
		\nu_o({\rm d} \omega) \\
		&\qquad\qquad=
		\sum_{y, y' \in V_g} f(y) \overline{g(y')}
		\int_\Omega \chi_0(h(y, y'; \omega))^{\frac12} e^{\sprod{i z}{h(y, y'; \omega)}} \nu_y({\rm d} \omega) \\
		&\qquad\qquad=
		\sum_{y, y' \in V_g} f(y) \overline{g(y')} P_{\sigma(y, y')}(i z).
	\end{align*}
	Consequently, by \cite[Lemma 6.1]{park2}
	\begin{align*}
		\int_{\scrD} \int_\Omega \calH f(i z, \omega) \calH \overline{g}(i z, \omega) \nu({\rm d} \omega)
		\pi({\rm d} z)
		&=
		\sum_{y, y' \in V_g} f(y) \overline{g(y')}
		\int_\scrD P_{\sigma(y, y')}(i z) \pi({\rm d} z) \\
		&=
		\sum_{y \in V_g} f(y) \overline{g(y)}
	\end{align*}
	which completes the proof. 
\end{proof}

In view of Theorem \ref{thm:12}, the Helgason transform can be defined for $f \in \ell^2(V_g)$ as a limit in 
$L^2(\scrD \times \Omega, \pi \times \nu)$ of $\calH f_n$ where $(f_n)$ is any sequence of finitely supported functions
convergent in $\ell^2(V_g)$ to $f$. 

\begin{proposition}
	\label{prop:3}
	Let $K$ be a radial function in $\ell^1(V_g)$. Then for each $f \in \ell^2(V_g)$, we have
	\[
		\calH (f \times K)(z, \omega) = \calH K(z) \calH f(z, \omega)
	\]
	for all $z \in \scrD$ and $\omega \in \Omega$.
\end{proposition}
\begin{proof}
	Suppose that $K$ is compactly supported. Then by Proposition \ref{prop:4}, for each $y \in V_g$ we have
	\begin{align*}
		\sum_{x \in V_g} K(\sigma(y, x)) \chi_0(h(y, x;\omega))^{\frac{1}{2}}
		e^{-\sprod{z}{h(y, x; \omega)}}
		&=
		\sum_{\lambda \in P^+} K(\lambda) 
		\sum_{x \in V_{\lambda}(y)} \chi_0(h(y, x;\omega))^{\frac{1}{2}} e^{-\sprod{z}{h(y, x; \omega)}} \\
		&=
		\sum_{\lambda \in P^+} K(\lambda) N_\lambda P_\lambda(-z) \\
		&=
		\sum_{x \in V_g} K(x) \vphi_{-z}(x) = \calH K(z)
	\end{align*}
	where in the last equality we have used Corollary \ref{cor:3}. Therefore, for each $f$ a compactly supported function 
	on $V_g$, $z \in \mathfrak{a}_\CC$ and $\omega \in \Omega$, we have
	\begin{align*}
		\calH (f \times K)(z, \omega) 
		&= \sum_{x \in V_g} (f \times K)(x) 
		\chi_0(h(o,x;\omega))^{\frac{1}{2}}e^{-\sprod{z}{h(o, x; \omega)}} \\
		&= \sum_{x \in V_g} \sum_{y \in V_g} f(y) K(\sigma(y, x)) 
		\chi_0(h(o,x;\omega))^{\frac{1}{2}}e^{-\sprod{z}{h(o, x; \omega)}} \\
		&=
		\sum_{y \in V_g} f(y) \chi_0(h(o,y;\omega))^{\frac{1}{2}} e^{-\sprod{z}{h(o, y; \omega)}}
		\sum_{x \in V_g} K(\sigma(y, x)) \chi_0(h(y, x;\omega))^{\frac{1}{2}} 
		e^{-\sprod{z}{h(y, x; \omega)}} \\
		&=
		\calH K(z) \calH f(z, \omega).
	\end{align*}
	Finally, let us observe that $f \times K \in \ell^2(V_g)$. Indeed, for any $g \in \ell^2(V_g)$, we can write
	\begin{align*}
		\sprod{f \times K}{g} 
		&\leqslant
		\sum_{x \in V_g} \sum_{y \in V_g} |f(y)| |K(\sigma(y, x))| |g(x)| \\
		&=
		\sum_{x \in V_g} \sum_{y \in V_g} \big(|f(y)|^2 |K(\sigma(y, x))|\big)^{\frac12} 
		\big(|g(x)|^2 |K(\sigma(y, x))|\big)^{\frac12} \\
		&\leqslant
		\|f\|_{\ell^2} \|g\|_{\ell^2} \|K\|_{\ell^1}.
	\end{align*}
	Hence, by a density argument, we complete the proof in the case where $K$ is radial and compactly supported. This result can 
	then be extended to all radial functions in $\ell^1(V_g)$ another application of a density argument.
\end{proof}

\begin{proposition}
	\label{prop:2}
	Let $f$ be a finitely supported function on $V_g$. Then for each $z \in \CC^r$,
	\begin{equation}
		\label{eq:53}
		(f \times \vphi_{-z})(x)
		=
		\int_{\Omega} \chi_0(h(o,x;\omega))^{\frac{1}{2}}e^{\sprod{z}{h(o, x; \omega)}}\calH f(z, \omega)
		\nu(\mathrm{d}\omega), \qquad x \in V_g.
	\end{equation}
	Moreover, if $f$ is a radial function belonging to $\ell^1(V_g)$ then for each $z \in [-1, 1]\eta + iU_0$,
	\[
		f \times \vphi_{-z} (x) = \vphi_z (x) \calH f(z), \qquad x \in V_g.
	\]
\end{proposition}
\begin{proof}
	To prove \eqref{eq:53}, we use \eqref{eq:2} to write
	\begin{align*}
		(f \times \vphi_{-z})(x)
		&=
		\sum_{y \in V_g} f(y) P_{\sigma(y, x)}(-z) \\
		&=
		\int_{\Omega} \chi_0(h(o,x;\omega))^{\frac{1}{2}} e^{\sprod{z}{h(o, x; \omega)}} 
		\sum_{y\in V_g} f(y)  \chi_0\big(h(o, y; \omega)\big)^\frac{1}{2} e^{-\sprod{z}{h(o, y; \omega)}} \, 
		\nu(\mathrm{d}\omega)\\
		&=\int_{\Omega} \chi_0(h(o,x;\omega))^{\frac{1}{2}}e^{\sprod{z}{h(o, x; \omega)}}\calH f(z, \omega)
		\nu(\mathrm{d}\omega).
	\end{align*}
	Suppose that $f$ is a radial function belonging to $\ell^1(V_g)$. By Lemma \ref{lem:3}\eqref{en:2:2} and Lemma \ref{lem:2}, 
	if $z \in [-1, 1] \eta + iU_0$, and $y \in V_g$ then
	\[
		\big| \vphi_{-z}(y) \big| \leqslant \vphi_{-\Re z}(y) \leqslant 1.
	\]
	Thus $f \times \vphi_{-z}$ is well-defined. Now, in view of Fubini's theorem and Corollary \ref{cor:3},
	\begin{align*}
		(f \times \vphi_{-z})(x) 
		&= 
		\int_{\Omega} \chi_0(h(o,x;\omega))^{\frac{1}{2}}e^{\sprod{z}{h(o, x; \omega)}}\calH f(z) \nu(\mathrm{d}\omega) \\
		&=
		\calH f(z) \int_{\Omega} \chi_0(h(o,x;\omega))^{\frac{1}{2}} e^{\sprod{z}{h(o, x; \omega)}} \nu(\mathrm{d}\omega) \\
		&=
		\calH f(z) \vphi_z(x)
	\end{align*}
	where the last equality is a consequence of \eqref{eq:9}. This completes the proof.
\end{proof}

\begin{theorem}
	\label{thm:14}
	Let $f$ be a finitely supported function on $V_g$. Then for each $x \in V_g$,
	\[
		f(x) =
		\int_{\scrD}
		\int_{\Omega} 
		\chi_0(h(o,x; \omega))^{\frac{1}{2}}e^{i\sprod{z}{h(o, x; \omega)}}\calH f(iz, \omega)
        \nu(\mathrm{d}\omega)
		\pi({\rm d} z).
	\]
\end{theorem}
\begin{proof}
	Given $y \in V_g$, we define a radial function on $V_g$ by the formula
	\[
		f_{\text{rad}}(x) = \frac{1}{N_\lambda} \sum_{y' \in V_{\lambda}(y)} f(y'),
		\qquad x\in V_\lambda(o).
	\]
	By Corollary \ref{cor:3}, for $z \in \scrD$,
	\[
		\calH f_{\text{rad}}(i z) = \sum_{x \in V_g} f_{\text{rad}}(x) \vphi_{-i z}(x).
	\]
	On the other hand, we have
	\begin{align*}
		\calH f_{\text{rad}}(i z)
		&= \sum_{x \in V_g} \frac{1}{N_{\sigma(o, x)}} \sum_{y' \in V_{\sigma(o, x)} (y)} f(y') 
		\int_{\Omega} \chi_0(h(o, x; \omega))^{\frac12} e^{-i\sprod{z}{h(o, x; \omega)}} \nu({\rm d} \omega) \\
		&= 
		\int_{\Omega} \sum_{y' \in V_g} f(y') 
		\frac{1}{N_{\sigma(y, y')}}
		\sum_{x \in V_{\sigma(y, y')}(o)}
		\chi_0(h(o, x; \omega))^{\frac12} e^{-i\sprod{z}{h(o, x; \omega)}} \nu({\rm d} \omega), \\
	\intertext{thus by Proposition \ref{prop:4}}
		&=
		\int_{\Omega} \sum_{y' \in V_g} f(y') \vphi_{-iz}(\sigma(y, y')) \nu({\rm d} \omega) \\
		&=
		f \times \vphi_{-iz}(y).
	\end{align*}
	Now, by Proposition \ref{prop:2},
	\[
		\calH f_{\text{rad}}(iz)
		=
		\int_{\Omega} 
		\chi_0(h(o,y; \omega))^{\frac{1}{2}}e^{i\sprod{z}{h(o, y; \omega)}}\calH f(iz, \omega)
        \nu(\mathrm{d}\omega).
	\]
	Since $f_{\text{rad}}$ is radial and finitely supported by \eqref{eq:21} we have 
	\[
		f(y) = f_{\text{rad}}(o) = \int_\scrD \calH f_{\text{rad}}(iz) \pi({\rm d} z)
	\]
	and the theorem follows.
\end{proof}

\subsection{Mass functions}
\label{sec:6}
In this section we introduce the notion of mass functions for finitely supported initial datum. We discuss their properties 
and prove our main $\ell^p$ convergence result concerning caloric functions.

For each $f$ a finitely supported function on good vertices $V_g$, and $p \in [1, \infty]$, we define its $p$-mass by 
the formula 
\begin{subequations}
	\begin{align}
	\label{eq:63a}
	M_p(f)(x)
	&=
	\sum_{y \in V_g}
	f(y)\fint_{\Omega(o,x)} \chi_{0}(h(o,y;\omega))^\frac{1}{p} \: \nu({\rm d} \omega), &\text{if } p \in [1, 2),
	\intertext{and}
	\label{eq:63b}
	M_p(f)(x)
	&=\frac{1}{\vphi_0(x)} f \times \vphi_0(x), &\text{if } p \in [2, \infty],
	\end{align}
\end{subequations}
where for a Borel subset $A \subseteq \Omega$, we have set
\[
	\fint_A F(\omega) \: \nu({\rm d} \omega) = \frac{1}{\nu(A)} \int_A F(\omega) \: \nu({\rm d} \omega).
\]
In fact for $p = 2$ both definitions of a mass function yield strong $\ell^2$ convergence for caloric functions, see 
Remark \ref{rem:1}.

Let us observe that if $f$ is \emph{finitely} supported then $M_p(f)$ is bounded for each $p \in [1, \infty]$. To see this, 
let us assume that
\[
	\supp f \subseteq B_R = \big\{x \in V_g : \abs{\sigma(o, x)} < R \big\}.
\]
By the Kostant's convexity theorem, \cite{Kostant}, we have
\[
	|h(o,y; \omega)| \leqslant |\sigma(o,y)| < R
\]
for all $y \in B_R$ and $\omega \in \Omega$. Therefore, if $p \in [1, 2)$, then we have
\begin{align*}
	|M_p(f)(x)|
	&\leqslant 
	\sum_{y \in B_R} |f(y)|\, 
	\fint_{\Omega(o,x)} e^{\frac{2}{p}\langle \eta, h(o, y;\omega)\rangle}\, \nu(\textrm{d}\omega) \\
	&\leqslant
	e^{R |\eta| \frac{2}{p}} \|f\|_{\ell^1}.
\end{align*}
On the other hand, by \eqref{eq:2} for $z = 0$ and \eqref{eq:17} we obtain
\begin{align*}
	\mathfrak{\Phi}(\sigma(y,x))
	&\leqslant
	\int_\Omega \chi_0(h(o,x; \omega))^{\frac12} \chi_0(h(o,y; \omega))^{\frac12} \nu({\rm d} \omega) \\
	&=
	\int_\Omega \chi_0(h(o,x; \omega))^{\frac12} e^{\sprod{\sigma(o, y)}{\eta}} \nu({\rm d} \omega)
	=
	\mathfrak{\Phi}(\sigma(o,x)) e^{\sprod{\sigma(o, y)}{\eta}},
\end{align*}
thus
\begin{equation}
	\label{eq:64}
	\frac{\mathbf{\Phi}(\sigma(y,x))}{\mathbf{\Phi}(\sigma(o,x))}
	\leqslant e^{\langle \eta, \sigma(o,y) \rangle}, 
\end{equation}
which in turn immediately implies that for $p \in [2, \infty]$ we have 
\[
	|M_p(f)(x)|\leqslant  e^{|\eta|R} \|f\|_{\ell^1}.
\]
In Section \ref{sec:5} we also study the spaces $\ell^1_{s_p}(V_g)$, $p \in [1, \infty]$, which are defined as
\begin{equation}
	\label{eq:72}
	\ell^1_{s_p}(V_g) =
	\left\{f : V_g \rightarrow \CC : \sum_{y \in V_g} |f(y)| \vphi_{s_p}(y) < \infty
	\right\}.
\end{equation}
where $s_p$ is given by the formula \eqref{eq:5}. Notice that $\ell^1_{s_p}(V_g)$ is a Banach space with a norm
\[
	\|f\|_{\ell^1_{s_p}} = \sum_{y \in V_g} |f(y)| \vphi_{s_p}(y).
\]
Let us observe that $\ell^1_{s_1}(V_g) = \ell^1(V_g)$. Moreover, if $f$ is radial and belongs to $\ell^1(V_g)$ then it belongs
to all $\ell^1_{s_p}(V_g)$ for $p \in (1, \infty]$. This is a consequence of Lemma \ref{lem:2} and Proposition \ref{prop:1}.
Furthermore, we can extend the definition of the mass function $M_p$ to radial functions belonging to $\ell^1_{s_p}(V_g)$. 
In fact, in this case the mass is a constant. To see this, let us first observe that by Corollary \ref{cor:3} the Helgason 
transform is constant. Next, if $p \in [1, 2)$, then by Proposition \ref{prop:2}, we have
\begin{align*}
	M_p(f)(x) 
	&= \sum_{y \in V_g} f(y) \fint_{\Omega(o, x)} \chi_0(h(o, y; \omega))^{\frac1p} \nu({\rm d} \omega) \\
	&= \fint_{\Omega(o, x)} \sum_{y \in V_g} f(y) \chi_0(h(o, y; \omega))^{\frac1p} \nu({\rm d} \omega) \\
	&= \fint_{\Omega(o, x)} \calH f(s_p) \nu({\rm d} \omega) = \calH f(s_p).
\end{align*}
If $p \in [2, \infty]$, again by Proposition \ref{prop:2}, we
\[
	M_p(f)(x) 
	= \frac{1}{\vphi_0(x)} f \times \vphi_0(x)
	= \calH f(0) = \calH f(s_p).
\]

\subsection{Finitely supported initial datum}
In this section we discuss the $\ell^p$ asymptotic behavior of caloric functions. Let $k$ be the transition function of 
an aperiodic isotropic finite range random walk on good vertices of an affine building $\scrX$. The caloric function 
$u: \NN \times V_s \rightarrow \CC$ with initial datum being a finitely supported function $f$, is the unique solution to
\begin{equation}
	\label{eq:71}
	\begin{aligned}
		u(n+1; x) &= \sum_{y \in V_g} k(y, x) u(n; y), \\
		u(0; x) &= f(x), \quad (n, x) \in \NN \times V_g.
	\end{aligned}
\end{equation}
In fact, the solution has a form
\[
	u(n; x) = \sum_{y \in V_g} k(n; y, x)f(y), \qquad (n, x) \in \NN_0 \times V_g.
\]
We often write $u_n(x) = u(n; x)$. Our aim in this section is to prove Theorem \ref{thm:6}. 

Let us first discuss the asymptotic behavior of caloric functions with finitely supported initial datum on the critical regions
$\scrN_n^p$ which are defined by \eqref{eq:4a}, \eqref{eq:4b} and \eqref{eq:4c}.
\begin{theorem}
	\label{thm:7}
	Let $k_n$ be a transition kernel of an admissible random walk on good vertices of an affine building $\scrX$.
	Let $p \in [1, \infty]$. If $u$ is the solution to \eqref{eq:71} where $f$ is a finitely supported function on $V_g$,
	then
	\[
		\frac{1}{\|k_n\|_{\ell^p}}
		\left(
		\sum_{x \in \scrN_n^p} 
		\big|u(n; x) - M_p(f) (x)\, k(n; x) \big|^p
		\right)^{\frac{1}{p}}
		=
		\|f\|_{\ell^1}
		\calO\big( \varepsilon_n'\big)
	\] 
	where
	\[
		\varepsilon_n' =
		\begin{cases}
			n^{-\frac{1}{2} + \gamma} &\text{if } p \in [1, 2], \\
			n^{-1} r_n &\text{if } p \in (2, \infty],
		\end{cases}
	\]
	The implied constant depends on the size of the support of $f$.
\end{theorem}
\begin{proof}
	Assume that the support of $f$ is contained in $B_R = \{x \in V_g: \abs{\sigma(o, x)} < R\}$, $R > 0$.
	We treat the cases $p \in [1, 2)$ and $p \in [2, \infty]$ separately.

	\vspace*{1ex}
	\noindent
	{Case {\bf I:} $p \in [1, 2)$.}
	For $x \in \scrN_n^p$, we write
	\begin{align*}
		&\sum_{y\in V_g} k(n; y, x) \,f(y)-M_p(f)(x)\, k(n; o, x) \\
		&\qquad\qquad
		=k(n;o,x) 
		\sum_{y\in B_R}
		f(y) \left( \frac{k (n;y,x)}{k(n;o,x)}-\fint_{\Omega(o,x)}\chi_0(h(o,y;\omega))^{\frac{1}{p}}\, 
		\nu({\rm d} \omega) \right).
	\end{align*}
	Since
	\begin{equation}
		\label{eq:32}
		\bigg|\frac{\sigma(o, x)}{n} - \delta_p\bigg| \leqslant n^{-\frac12+\gamma},
	\end{equation}
	we can use the quotient asymptotics from Corollary \ref{cor:2}. We obtain
	\[
		\frac{k(n;y,x)}{k(n;o,x)}
		=\chi_{0}(\sigma(o,x)-\sigma(y,x))^{\frac{1}{p}} \Big(1+\calO\big(n^{-1/2+\gamma}\big)\Big).
	\]
	By \eqref{eq:12},
	\[
		\sprod{\sigma(o,x)-\sigma(y,x)}{\eta} \leqslant \sprod{\sigma(o, y)}{\eta} \leqslant R|\eta|,
	\]
	therefore
	\[
		\frac{k(n;y,x)}{k(n;o,x)} = \chi_{0}(\sigma(o,x)-\sigma(y,x))^{\frac{1}{p}} + \calO\big(n^{-\frac12+\gamma}\big).
	\]
	On the other hand, by \eqref{eq:32}
	\[
		\sprod{\sigma(o, x)}{\alpha} \geq 2 R
	\]
	for all $\alpha \in \Phi^+$, and sufficiently large $n$. Thus for each $\omega \in \Omega(o, x)$  we have
	$x \in [y, \omega]$, hence by \eqref{eq:4},
	\[
		h(o,y; \omega)=\sigma(o,x)-\sigma(y, x).
	\]
	Therefore,
	\begin{align*}
		\fint_{\Omega(o,x)} \chi_0(h(o,y;\omega))^{\frac1p} \, \nu({\rm d}\omega)
		&=\fint_{\Omega(o,x)}\chi_{0}(\sigma(o, x)-\sigma(y, x))^{\frac1p} \, \nu({\rm d}\omega) \\
		&=\chi_{0}(\sigma(o,x)-\sigma(o,y))^{\frac1p},
	\end{align*}
 	and consequently,
	\begin{align*}
		\sum_{x \in \scrN_n^p}
		\left|
		u(n; x) 
		- M_p(f)(x)\, k(n;x)
		\right|^p
		&=
		\left(\sum_{x \in \scrN_n^p}
		k(n;o,x)^p\right) \|f\|_{\ell^1}^p \calO\Big(n^{p(-\frac12 + \gamma)}\Big) \\
		&\leqslant
		\|k_n\|_{\ell^p}^p \|f\|_{\ell^1}^p \calO\Big(n^{p(-\frac12 + \gamma)}\Big), 
	\end{align*}
	which completes the proof in this case.
	
	\vspace*{1ex}
	\noindent
	{Case {\bf II:} $p \in [2, \infty]$.} 
	For $x \in \scrN_n^p$, we have $|\sigma(o,x)|=\calO(\tilde{r}_n)$, where
	\[
		\tilde{r}_n=
		\begin{cases}
			n^{\frac{1}{2} + \gamma} &\text{if } p =2 \\
			r_n &\text{if } p \in (2, \infty],
		\end{cases}
	\]
	here we have used the notation introduced in \eqref{eq:40}. Thus by Theorem \ref{thm:5}, we have
	\begin{align*}
		&\sum_{y\in V_g} k(n;y,x)\,f(y) - M_p(f)(x)\, k(n;o,x) \\
		&\qquad\qquad=
		k(n; o, x) 
		\sum_{y \in B_{R}} f(y) 
		\left( \frac{k(n;y,x)}{k(n;o,x)} - \frac{\mathbf{\Phi}(\sigma(y,x))}{\mathbf{\Phi}(\sigma(o,x))} \right)\\
		&\qquad\qquad=
		k(n; o, x) 
		\sum_{y \in B_{R}} f(y)
		\left( \frac{\mathbf{\Phi}(\sigma(y,x))}{\mathbf{\Phi}(\sigma(o,x))}\big(1+\calO(\tilde{r}_n)\big)
		-
		\frac{\mathbf{\Phi}(\sigma(y,x))}{\mathbf{\Phi}(\sigma(o,x))} \right).
	\end{align*}
	By \eqref{eq:64}, for all $y \in B_R$, the ratio
	\[
		\sup_{x \in V_g}
		\frac{\mathbf{\Phi}(\sigma(y,x))}{\mathbf{\Phi}(\sigma(o,x))} \leqslant e^{\sprod{\eta}{\sigma(o, y)}}, 
	\]
	thus
	\[
		u(n; x) - M_p(f)(x)\, k(n; x)
		=
		k(n;x) \|f\|_{\ell^1} \calO(\tilde{r}_n).
	\]
	If $p \in [2, \infty)$, by summing up over critical region $\scrN_n^p$, we get
	\begin{align*}
		\sum_{x \in \scrN_n^p} \left|u(n; x) - M_p(f)(x)\, k(n; x)\right|^p
		&=
		\bigg(\sum_{x \in \scrN_n^p} k(n; x)^p\bigg) \|f\|_{\ell^1}^p \calO(\tilde{r}_n) \\
		&\leqslant
		\|k_n\|_{\ell^p}^p  \|f\|_{\ell^1}^p \calO(\tilde{r}_n)
	\end{align*}
	which completes the proof for $p \in [2, \infty)$. Lastly, for $p = \infty$ the reasoning is similar.
\end{proof}
		
In the following result we describe how caloric functions with finitely supported initial datum behave outside the $\ell^p$ 
critical regions $\scrN_n^p$.
\begin{theorem} 
	\label{thm:8}
	Let $k$ be a transition kernel of an admissible random walk on good vertices of an affine building $\scrX$.
	Let $p \in [1, \infty]$. There is $c > 0$ such that for each $u$ satisfying \eqref{eq:71} with $f$ being a finitely
	supported function on $V_g$, we have
	\[
		\frac{1}{\| k_n \|_{\ell^p}}
		\left(
		\sum_{x \in V_g \smallsetminus \scrN_n^p} |u(n; x)|^p
		\right)^{\frac{1}{p}}
		= \|f\|_{\ell^1} \calO\big( \varepsilon_n''\big)
	\]
	where 
	\[
		\varepsilon_n'' =
		\begin{cases}
			e^{-cn^{2\gamma}} &\text{if } p \in [1, 2], \\
			n^{-\gamma|\Phi^{++}|} &\text{if } p=2,\\
			e^{-c r_n} &\text{if } p\in (2, \infty].
		\end{cases}
	\]
	The implied constant depends on size of the support of $f$.
\end{theorem}
\begin{proof}
	Assume that $f$ is supported on $B_R= \{y \in V_g: \, |\sigma(o,y)|< R\}$, for some $R>0$. If $p \in [1, \infty)$,
	then by the Minkowski inequality we have
	\begin{align*}
		\Bigg(\sum_{x \in V_g \setminus \scrN_n^p} 
		\big| \sum_{y \in B_{R}} k(n; y, x) \, f(y) \big|^p \bigg)^{\frac{1}{p}}
		&\leqslant 
		\sum_{y\in B_{R}}
		\bigg( \sum_{x \in V_g \setminus \scrN_n^p} k(n;y,x)^p\, |f(y)|^p \bigg)^{\frac{1}{p}} \notag\\
		&\leqslant
		\sum_{y\in B_{R}}|f(y)|\bigg( \sum_{x \in V_g \setminus \scrN_n^p} k(n;y,x)^p \bigg)^{\frac{1}{p}}. 
	\end{align*}
	Now, let us first consider the case $p \in [1, 2)$. Observe that if $x \in V_g \setminus \scrN_n^p$, then 
	\[
		\left|\sigma(o,x)-n\delta_p\right|>n^{\frac{1}{2}+\gamma},
	\]
	and hence by \eqref{eq:14}
	\begin{align*}
		\left|\sigma(y,x)-n\delta_p\right|
		&> 
		\left|\sigma(o,x)-n\delta_p\right| - \left|\sigma(y,x)-\sigma(o,x)\right|\\
		&> 
		n^{\frac{1}{2}+\gamma}-R.
	\end{align*}
	Therefore, if $n \geqslant 4R^2$, then $x \in V_g \setminus \tilde{\scrN}_{n, y}^p$ where
	\[
		\tilde{\scrN}_{n, y}^p
		=
		\left\{x \in V_g: |\sigma(y,x) - n \delta_p | \leqslant \tfrac{1}{2} n^{\frac{1}{2}+\gamma} \right\}.
	\]
	Next, by repeating the reasoning from the proof of Theorem \ref{thm:1}, one can show that there is $c > 0$
	such that for each $y \in B_R$,
	\[
		\sum_{x \in V_g \setminus \tilde{\scrN}_{n, y}^p} k(n;y,x)^p
		\lesssim
		\|k_n\|_{\ell^p}^p e^{-c p n^{2\gamma}} 
	\]
	where the implied constant is independent of $y \in B_R$. Hence, if $n \geqslant 4R^2$, then we get
	\begin{align*}
		\sum_{y\in B_{R}}|f(y)|\bigg( \sum_{x \in V_g \setminus \scrN_n^p} k(n;y,x)^p \bigg)^{\frac{1}{p}}
		&\leqslant
		\sum_{y \in B_R} |f(y)| \bigg( \sum_{x \in V_g \setminus \tilde{\scrN}_{n, y}^p} k(n;y,x)^p \bigg)^{\frac{1}{p}} \\
		&\lesssim
		\|k_n\|_{\ell^p} e^{-c n^{2\gamma}} \|f\|_{\ell^1} 
	\end{align*}
	which completes the proof for $p \in [1, 2)$.
	
	The case when $p \in [2, \infty]$ is treated in a similar way. Namely, if $p \in (2, \infty)$ we observe that for
	$x \in V_g \setminus \scrN_n^p$, we have
	\[
		|\sigma(y,x)| \geqslant |\sigma(o, x)| - |\sigma(o, y)| \geqslant r_n - R,
	\]
	thus by \eqref{eq:40}, there is $N_1 \geqslant 1$ such that $r_n \geqslant 2 R$ for all $n \geqslant N_1$. In particular, 
	for all for $n \geqslant N_1$ and $y \in B_R$, we have $x \in V_g \setminus \tilde{N}_{n, y}^p$ where
	\[
		\tilde{\scrN}_{n, y}^p
		=\left\{
		x \in V_g : |\sigma(y, x)| \geqslant \tfrac{1}{2} r_n
		\right\}.
	\]
	By repeating the arguments used in the proof of Theorem \ref{thm:3}, one can show that
	\[
		\sum_{x \in V_g \setminus \tilde{\scrN}_{n, y}^p}
		k(n; y, x)^p
		\lesssim
		\|k_n\|_{\ell^p}^p e^{-c p r_n}.
	\]
	Analogous reasoning applies to $p = \infty$.

	Lastly, we consider $p = 2$. We set
	\[
		\tilde{\scrN}_{n, y}^2
		=
		\left\{ x \in V_g : 2 n^{\frac{1}{2}-\gamma} 
		\leqslant 
		|\sigma(y, x)| 
		\leqslant 
		\tfrac{1}{2} n^{\frac{1}{2}+\gamma}, \text{ and }
		\sprod{\sigma(y, x)}{\alpha} \geqslant 2 n^{\frac{1}{2} - \gamma'}
   		\text{ for all } \alpha\in \Phi^{+}\right\}.
	\]
	If $x \in \tilde{\scrN}_{n, y}^2$, then by \eqref{eq:83}, we have
	\begin{align*}
		n^{\frac{1}{2}-\gamma} \leqslant 2 n^{\frac{1}{2}-\gamma} - R 
		\leqslant |\sigma(y, x)| - |\sigma(o, y)|
		\leqslant
		|\sigma(o, x)|
	\intertext{and \eqref{eq:14}}
		|\sigma(o, x)| 
		\leqslant |\sigma(o, y)| + |\sigma(y, x)| \leqslant R + \frac{1}{2} n^{\frac{1}{2}+\gamma} 
		\leqslant n^{\frac{1}{2}+\gamma}.
	\end{align*}
	Finally, for each $\alpha \in \Phi^{+}$, by \eqref{eq:11},
	\begin{align*}
		\sprod{\sigma(o, x)}{\alpha} 
		\geqslant \sprod{\sigma(y, x)}{\alpha} - \max_{\alpha \in \Phi^+} |\alpha| |\sigma(o, y)|
		\geqslant n^{\frac{1}{2} - \gamma'}.
	\end{align*}
	Hence, there is $N_1 \geqslant 1$ such that for all $n \geqslant N_1$ and $y \in B_R$,
	$V_g \setminus \scrN_n^2 \subset V_g \setminus \tilde{\scrN}_{n, y}^2$.
	Now, following the arguments in the proof of Theorem \ref{thm:2}, one can show that for each $y \in B_R$,
	\[
		\sum_{x \in V_g \setminus \tilde{\scrN}_{n, y}^p} k(n;y,x)^2
		\lesssim
		\|k_n\|_{\ell^2}^2 n^{-2 \gamma |\Phi^{++}|} 
	\]
	where the implied constant is independent of $y \in B_R$. This completes the proof of the theorem.
\end{proof}

We are now in a position to prove the main theorem of this section.
\begin{theorem}
	\label{thm:6}
	Let $k_n$ be a transition kernel of an admissible random walk on good vertices of an affine building $\scrX$.
	Let $p \in [1, \infty]$. Then for each $u$ satisfying \eqref{eq:71} with $f$ being a finitely
    supported function on $V_g$, we have
	\[
		\frac{1}{\|k_n\|_{\ell^p}}
		\left\| u_n - M_p(f) k_n \right\|_{\ell^p}
		=
		\Big(\|f\|_{\ell^1} + \|M_p(f)\|_{\ell^\infty}\Big)
		\calO\big(\varepsilon_n \big)
	\]
	where 
	\[
		\varepsilon_n =
		\begin{cases}
			n^{-\frac{1}{2}+\gamma} &\text{if } p \in [1, 2), \\
			n^{-\frac{|\Phi^{++}|}{2 |\Phi^{++}|+2}} &\text{if } p=2,\\
			n^{-1}r_n &\text{if } p\in (2, \infty].
		\end{cases}
	\]
	The implied constant depends on the size of the support of $f$.
\end{theorem}
\begin{proof}
	Using the triangle inequality, we can write
	\begin{align*}
		\big\|u_n - M_p(f) k_n \big\|_{\ell^p(V_g \setminus \scrN_n^p)}
		\leqslant
		\big\|u_n \big\|_{\ell^p(V_g \setminus \scrN_n^p)}
		+
		\|M_p(f)\|_{\ell^\infty}
		\big\| k_n \big\|_{\ell^p(V_g \setminus \scrN_n^p)}^p.
	\end{align*}
	Since the mass function $M_p(f)$ is bounded, Theorem \ref{thm:8}, together with Theorems \ref{thm:1}, \ref{thm:2} and 
	\ref{thm:3}, implies that
	\[
		\big\|u_n - M_p(f) k_n \big\|_{\ell^p(V_g \setminus \scrN_n^p)}
		=
		\|k_n\|_{\ell^p}
		\big(
		\|f\|_{\ell^1} + \|M_p(f)\|_{\ell^\infty}
		\big)
		\calO\big(\varepsilon''_n \big).
	\]
	By Theorem \ref{thm:7}, we have
	\[
		\big\|u_n-M_p(f) k_n\big\|_{\ell^p(\scrN_n^p)}
		= \|k_n\|_{\ell^p} \|f\|_{\ell^1} \calO\big(\varepsilon'_n)
	\]
	hence to make $\varepsilon''_n = \calO\big(\varepsilon'_n)$, in the case $p = 2$, we need to take
	\[
		\gamma = \frac{1}{2(|\Phi^{++}|+1)}
	\]
	which finishes the proof of the theorem.
\end{proof}

\begin{remark}
	\label{rem:1}
	Given $f$ a finitely supported function on $V_g$, we can set
	\[
		\tilde{M}_2(f)(x) =\sum_{y\in V_s}f(y)\fint_{\Omega(o,x)}\chi_{0}(h(o,y;\omega))^{\frac{1}{2}}\, \nu({\rm d}\omega),
		\quad x \in V_g.
	\]
	Following the arguments as in the proof of Theorem \ref{thm:7} for $p \in [1, 2)$ but this time in $\ell^2$, one can show 
	that 
	\[
		\lim_{n \to \infty}
		\frac{1}{\|k_n\|_{\ell^2}}
		\left(\sum_{x \in \scrN_n^2} \big|u_n(x) - \tilde{M}_p(f)(x) k_n(x)\big|^2\right)^{1/2}
		=0,
	\]
	which together with Theorem \ref{thm:2} and Theorem \ref{thm:8} implies again
	\[
		\lim_{n \to \infty}
		\frac{1}{\|k_n\|_{\ell^2}}
		\left\| u_n - M_p(f) k_n \right\|_{\ell^2} = 0.
	\]
	The details are left to the interested reader.
\end{remark}

\subsection{Radial $\ell^1$ datum and beyond}
\label{sec:5}
In this section we prove $\ell^p$ convergence of caloric functions for larger classes of initial datum than being just
finitely supported. We first discuss the case of radial functions in $\ell^1(V_g)$. 

Let us recall that if $f$ is radial function in $\ell^1(V_g)$ then it belongs to all $\ell^1_{s_p}(V_g)$, $p \in [1, \infty]$,
see \eqref{eq:72} for the definition. In Section \ref{sec:6} we have already seen that for a radial function $f$ in
$\ell_{s_p}^1(V_g)$ the mass function is constant
\[
	M_p(f) \equiv \calH f(s_p).
\]
\begin{theorem}
	\label{thm:11}
	Let $k_n$ be a transition kernel of an admissible random walk on good vertices of an affine building $\scrX$.
	Let $p \in [1, \infty]$. Then for each $u$ satisfying \eqref{eq:71} with $f$ being a radial function in $\ell^1(V_g)$,
	we have
	\[
		\lim_{n \to \infty} \frac{1}{\|k_n\|_{\ell^p}}
		\left\| u_n - M_p(f) k_n \right\|_{\ell^p}
		=0.
	\]
\end{theorem}
\begin{proof}
	Let $\epsilon > 0$. Let $g$ be a radial finitely supported function on $V_g$ such that
	\[
		\|f-\widetilde{f}\|_{\ell^1} < \sum_{y\in V_s}|f(y)-\widetilde{f}(y)|\, <\tfrac{1}{3} \epsilon.
	\]
	Given $p \in [1, \infty]$, by Lemma \ref{lem:2} and Proposition \ref{prop:1}, there is $C_p \geqslant 1$, such that
	for all $y \in V_g$,
	\[
		0 \leqslant \vphi_{s_p}(y) \leqslant C_p.
	\]
	Hence,
	\begin{align*}
		|M_p(f) - M_p(g)| 
		&\leqslant \sum_{y \in V_g} |f(y) - g(y)| \vphi_{s_p}(y) \\
		&\leqslant \tfrac{1}{3} C_p \epsilon.
	\end{align*}
	Let $\tilde{u}$ be the caloric function with initial datum $g$, that is
	\[
		\tilde{u}(n; x) = \sum_{y \in V_g} k(n; y, x) g(y).
	\]
	By Theorem \ref{thm:6}, there is $N_1 \geqslant 1$ such that for all $n \geqslant N_1$,
	\[
		\frac{1}{\|k_n\|_{\ell^p}} \big\|\tilde{u}(n; \cdot) - M_p(g) k(n; o, \cdot) \big\|_{\ell^p}
		\leqslant \tfrac{1}{3} \epsilon.
	\]
	Hence,
	\begin{align*}
		\|u(n;\cdot)-M_p(f) p(k; o,\cdot)\|_{\ell^p}
		&\leqslant
		\|u(n;\cdot)-\tilde{u}(n;\cdot)\|_{\ell^p} 
		+\|M_{p}(f) k(n;o, \cdot) - M_p(\tilde{f}) k_n \|_{\ell^p} \\
		&\phantom{\leqslant\|u(n;\cdot)-\tilde{u}(n;\cdot)\|_{\ell^p}\ }
		+\|\tilde{u}(n;\cdot)-M_p(\tilde{f}) k_n \|_{\ell^p} \\
		&\leqslant
		\Big(\tfrac{1}{3} \epsilon + \tfrac{1}{3} C_p \epsilon + \tfrac{1}{3}\Big) \|k_n\|_{\ell^p}
	\end{align*}
	which completes the proof.
\end{proof}

In the following theorem, we extend the previous result to radial functions in $\ell^1_{s_p}(V_g)$. 
The proof draws on both the Kunze--Stein phenomenon and Herz’s \emph{principe de majoration}, tools that are available only 
when the underlying building is of Bruhat--Tits type. While the Kunze--Stein phenomenon can be justified for radial convolutors
even in the case of exotic buildings, we do not see how to establish Herz’s \emph{principe de majoration} in that setting.
\begin{theorem}
	\label{thm:9}
	Let $\mathbf{G}$ be a simply connected semisimple algebraic group defined over a locally compact non-Archimedean field,
	and let $\scrX$ be its Bruhat--Tits building. Let $k_n$ be a transition kernel of an admissible random walk on good vertices
	of $\scrX$. Let $p \in [1, \infty]$. Then for each $u$ satisfying \eqref{eq:71} with $f$ being a radial function
	in $\ell^1_{s_p}(V_g)$, we have
	\[
		\lim_{n \to \infty} 
		\frac{1}{\|k_n \|_{\ell^p}}
		\big\|u_n-M_p(f) \, k_n\|_{\ell^p} = 0.
	\]
\end{theorem}
\begin{proof}
	First, let us consider $p \in [1, 2]$. We use a density argument to lift the convergence result of Theorem \ref{thm:6}. 
	Let $f \in \ell^1(V_g)$ be a radial function. Let $\epsilon > 0$. Let $\tilde{f}$ be a radial and finitely supported 
	function such that
	\[
		\sum_{y \in V_g} |f(y) - \tilde{f}(y)| \vphi_{s_p}(y) \leqslant \tfrac{1}{3} \epsilon.
	\]
	Since
	\begin{align*}
		\big| M_p(f)-M_p(\tilde{f}) \big| 
		&\leqslant 
		\sum_{y \in V_g} \big|f(y) - \tilde{f}(y)\big| \vphi_{s_p}(y) = \calH(|f - \tilde{f}|)(s_p) ,
	\end{align*}
	we have
	\[
		\big\|M_p(f) k_n - M_p(\tilde{f}) k_n\|_{\ell^p} \leqslant \tfrac{1}{3} \epsilon \|k_n\|_{\ell^p}.
	\]
	Let $\tilde{u}$ be the caloric function with initial datum $\tilde{f}$.
	Then by Theorem \ref{thm:6}, there is $N_1 \geqslant 1$ such that for all $n \geqslant N_1$,
	\[
		\frac{1}{\|k_n\|_{\ell^p}} \,\|\tilde{u}(n; \cdot)-M_p(\tilde{f}) \, k(n;o, \cdot) \|_{\ell^p}
		\leqslant \tfrac{1}{3} \epsilon.
	\]
	Since $k_n$ is radial, by the radial version of Herz's principle, see Theorem \ref{thm:A:2}, we get
	\[
		\big\|u(n;\cdot)-\tilde{u}(n;\cdot)\big\|_{\ell^p}
		\leqslant
		\calH (|f - \tilde{f}|)(s_p) \|k_n\|_{\ell^p} 
		\leqslant 
		\tfrac{1}{3} \epsilon \|k_n\|_{\ell^p}.
	\]
	Therefore, for $n \geqslant N_1$, we obtain
	\begin{align*}
		\big\|u(n; \cdot) - M_p(f) k(n; \cdot) \big\|_{\ell^p}
		&\leqslant
		\big\|u(n;\cdot)-\tilde{u}(n;\cdot)\big\|_{\ell^p}
		+\big\|M_{p}(f)\, k_n - M_p(\tilde{f}) k_n\big\|_{\ell^p} \\
		&\phantom{\leqslant \big\|u(n;\cdot)-\tilde{u}(n;\cdot)\big\|_{\ell^p} \ }
		+\big\|\tilde{u}_n - M_p(\tilde{f}) k_n \|_{\ell^p} \\
		&\leqslant \epsilon \|k_n\|_{\ell^p},
	\end{align*}
	which finishes the proof.

	Assume that $p \in (2, \infty]$. Let $m = n - \lceil \frac{n}{2}\rceil$. Since
	\[
		k(n; x, y) = \sum_{z \in V_g} k(m; x, z) k(n-m; z, y),
	\]
	we have
	\begin{align*}
		u(n; x) 
		&= \sum_{y \in V_g} f(y) k_n(y, x) \\
		&= \sum_{y \in V_g} f(y) \sum_{z \in V_g} k_{n-m}(y, z) k_{m}(z, x) \\
		&= \sum_{z \in V_g} u(n-m; z) k_{m}(z, x).
	\end{align*}
	Hence,
	\begin{align*}
		u(n; x) - M_p(f) k(n; x) = \sum_{z \in V_g} \big( u(n-m; z) - M_p(f) k_{n-m}(z)\big) k_m(z, x).
	\end{align*}
	Therefore, by the Kunze--Stein phenomenon, see Corollary \ref{cor:A:2}, we get 
	\begin{align*}
		\left\| u(n; \cdot) - M_p(f) k(n; \cdot) \right\|_{\ell^p}
		&=
		\left\|  \sum_{z \in V_g} \big( u(n-m; z) - M_p(f) k_{n-m}(z)\big) k_m(z, x) \right\|_{\ell^p(x)} \\
		&\leqslant
		C_p
		\left\|  u(n-m; \cdot) - M_p(f) k_{n-m}(\cdot) \right\|_{\ell^2}
		\|k_{m}\|_{\ell^2}.
	\end{align*}
	Since
	\[
		m = 
		\frac{1}{2} 
		\begin{cases}
			n &\text{if $n$ is even,} \\
			n-1 &\text{if $n$ is odd,}
		\end{cases}
	\]
	in view of Theorems \ref{thm:2} and \ref{thm:3}, for all $p \in [2, \infty]$,
	\begin{align*}
		\|k_m\|_{\ell^2}^2 
		&\approx m^{-\frac{r}{2} - |\Phi^{++}|} \varrho^{2m} \\
		&\approx n^{-\frac{r}{2} - |\Phi^{++}|} \varrho^n
		\approx \|k_n\|_{\ell^p},
	\intertext{and}
		\|k_{n-m}\|_{\ell^2}^2
		&\approx
		(n-m)^{-\frac{r}{2} - |\Phi^{++}|} \varrho^{2(n-m)}\\
		&\approx m^{-\frac{r}{2} - |\Phi^{++}|} \varrho^{2m}
	    \approx \|k_m\|_{\ell^2}^2.
	\end{align*}
	Hence, by the first part of the proof, for $p = 2$ we conclude that
	\[
		\lim_{n \to \infty} 
		\frac{1}{\|k_n\|_{\ell^p}}
		\left\|  u(n-m; \cdot) - M_p(f) k_{n-m}(\cdot) \right\|_{\ell^2} \|k_m\|_{\ell^2} = 0.
	\]
	The proof is now complete.
\end{proof}

Lastly, we study the $\ell^p$-convergence of the caloric functions with boundary datum from a class of functions which are not 
necessarily radial nor finitely supported. To this end, for each $p \in [1, \infty]$ we introduce the following weights,
\[
	w_p(y)=\begin{cases}
		e^{\frac{2}{p} \sprod{\sigma(o,y)}{\eta}} &\text{if } p \in [1, 2), \\
		e^{ \sprod{\sigma(o,y)}{\eta}} &\text{if } p \in [2, \infty],
	\end{cases}
\]
where $y \in V_g$. Since $w_p(y) \geqslant 1$, we have
\[
	\ell^1(w_p)=\left\{f:V_g\rightarrow \mathbb{C}:\quad \sum_{y\in V_g}|f(y)|\,w_p(y)<+\infty \right\}\subseteq \ell^1.
\]
Moreover, if $f \in \ell^1(w_p)$ then the mass function $M_p(f)$ is well-defined and bounded. 
Indeed, if $p \in [1, 2)$, then by \eqref{eq:17} we have 
\begin{align*}
	|M_p(f)(x)|
	&\leqslant
	\sum_{y\in V_g} |f(y)| 
	\fint_{\Omega(o,x)} e^{\frac{2}{p} \sprod{\eta}{h(o, y;\omega)}} \, \nu({\rm d}\omega) \\
	&=\sum_{y\in V_g} |f(y)|\, e^{\frac{2}{p} \sprod{\sigma(o, y)}{\eta}} < \infty.
\end{align*}
On the other hand, if $p \in [2, \infty]$, then by \eqref{eq:51}, \eqref{eq:60} and \eqref{eq:64} we have
\begin{align*}
	|M_p(f)(x)|
	&\leqslant
	\sum_{y \in V_g} |f(y)| \frac{\mathbf{\Phi}(\sigma(y, x))}{\mathbf{\Phi}(\sigma(y, x))} \\
	&\leqslant
	\sum_{y\in V_s} |f(y)|\, e^{ \sprod{\sigma(o,y)}{\eta}}.
\end{align*}
For $\ell^1(w_p)$ we have the following result.
\begin{theorem}
	\label{thm:10}
	Let $k_n$ be a transition kernel of an admissible random walk on good vertices of an affine building $\scrX$.
	Let $p \in [1, \infty]$. Then for each $u$ satisfying \eqref{eq:71} with $f \in \ell^1(w_p)$, we have
	\[
		\lim_{n \to \infty}
		\frac{1}{\|k_n \|_{\ell^p}}
		\big\|u_n-M_p(f) \, k_n\|_{\ell^p} = 0.
	\]
\end{theorem}
\begin{proof}
	Let $p \in [1, 2)$ and $\epsilon > 0$. Since $f w_p \in \ell^1(V_g)$, there is $g$ a finitely supported function on
	$V_g$ such that
	\begin{align*}
		\|f - g\|_{\ell^1} 
		&\leqslant \sum_{x \in V_g} |f(x) - g(x)| \\
		&\leqslant \sum_{x \in V_g} |f(x) - g(x)| w_p(x) 
		\leqslant \tfrac{1}{3} \epsilon. 
	\end{align*}
	Moreover, by \eqref{eq:17}, for each $x \in V_g$,  we have
	\begin{align*}
		\big| M_p(f)(x) - M_p(g)(x)\big| 
		&\leqslant
		\sum_{y \in V_g} |f(x) - g(x)| 
		\fint_{\Omega(o, x)} e^{\frac{2}{p}\sprod{\eta}{h(o,y;\omega)}}\, \nu({\rm d}\omega) \\
		&\leqslant  \sum_{y\in V_g} |f(y)-g(y)|\, e^{\frac{2}{p}\sprod{\sigma(o,y)}{\eta}}
		< \tfrac{1}{3}\epsilon
	\end{align*}
	thus
	\[
		\big\|M_p(f) k_n - M_p(g) k_n \big\|_{\ell^p}
		\leqslant
		\tfrac{1}{3}\epsilon \|k_n\|_{\ell^p}.
	\]
	Let $\tilde{u}$ denote the caloric function with initial datum $g$. Then by Theorem \ref{thm:6}, there is $N_1 \geqslant 1$,
	such that for all $n \geqslant N_1$,
	\[
		\frac{1}{\|k_n\|_{\ell^p}}
		\big\|\tilde{u}(n; \cdot)-M_p(g) k_n\big\|_{\ell^p}
		\leqslant \tfrac{1}{3} \epsilon.
	\]
	Using the Minkowski's inequality for integrals, we get
	\begin{align*}
		\big\|u(n; \cdot) - \tilde{u}(n; \cdot) \big\|_{\ell^p}
		&=
		\bigg\|\sum_{y \in V_g} (f(y)-g(y)) k(n; y, \cdot) \bigg\|_{\ell^p} \\
		&\leqslant
		\|f -g\|_{\ell^1} \|k_n\|_{\ell^p} 
		\leqslant \tfrac{1}{3} \epsilon  \|k_n\|_{\ell^p}.
	\end{align*}
	Hence, for $n \geqslant N_1$,
	\begin{align*}
		\big\|u(n;\cdot)-M_p(f) k_n\|_{\ell^p}
		&\leqslant
		\big\|u(n; \cdot) - \tilde{u}(n; \cdot) \big\|_{\ell^p}
		+
		\big\|\tilde{u}(n; \cdot)-M_p(g) k_n \big\|_{\ell^p}
		+
		\big\|M_p(f)  k_n - M_p(g) k_n \big\|_{\ell^p}
		\\
		&\leqslant \epsilon \|k_n\|_{\ell^p}.
	\end{align*}
	The proof for $p \in [2, \infty]$ follows by the same line of reasoning provided one uses \eqref{eq:64} in place of 
	\eqref{eq:17}.
\end{proof}
	
\appendix

\section{Herz's principle}
\label{sec:8}
In this appendix we discuss the radial version of Herz's \emph{principe de majoration} under the assumption that the building
is the Bruhat--Tits building, that is the affine building associated to a simply connected semisimple algebraic group defined
over a locally compact non-Archimedean field. Then we deduce the Kunze--Stein phenomenon. Both facts were proved in \cite{veca}, 
however in this paper we just need their radial versions and they have much simpler proofs.

\begin{theorem}[Herz's \emph{principe de majoration}]
	\label{thm:A:2}
	Let $\mathbf{G}$ be a simply connected semisimple algebraic group defined over a locally compact non-Archimedean field,
	and let $\scrX$ be its Bruhat--Tits building. Let $p \in (1, 2]$. Suppose that $K$ is a radial function in 
	$\ell^1_{s_p}(V_g)$. Then for each radial $f \in \ell^p(V_g)$,
	\[
		\|f \times K\|_{\ell^p} \leqslant \calH (|K|)(s_p) \|f\|_{\ell^p}.
	\]
\end{theorem}
\begin{proof}
	We first show that for all radial functions $f \in \ell^p(V_g)$ and $h \in \ell^{p'}(V_g)$ and $x \in V_g$,
	\begin{equation}
		\label{eq:76}
		|f \times h(x)| \leqslant \vphi_{s_p}(x) \|f\|_{\ell^p} \|h\|_{\ell^{p'}}.
	\end{equation}
	Let $x = g_0 . o$. We have
	\begin{align*}
		h \times f (x) 
		&= \sum_{y \in V_g} h(y) f(\sigma(y, x)) \\
		&= \int_G h(g.o) f(\sigma(g.o, g_0.o)) {\: \rm d} g\\
		&= \int_G h(g^{-1}.o) f(g g_0.o) {\: \rm d} g\\
		&= \int_K \int_{N A} h\big((nak)^{-1}.o\big) f(nak g_0.o) 
		\chi_0\big(-h(o, nak.o; \omega_0)\big) {\: \rm d} n {\: \rm d} a {\: \rm d} k \\
		&= \int_K \int_{N A} h\big(a^{-1}n^{-1}.o\big) f(nak g_0.o) 
		\chi_0\big(-h(o, a.o; \omega_0)\big) {\: \rm d} n {\: \rm d} a {\: \rm d} k.
	\end{align*}
	Now, for a fixed $k \in K$, we apply H\"older's inequality to get
	\begin{align*}
		&
		\bigg|
		\int_{N A} h\big(a^{-1}n^{-1}.o\big) f(nak g_0.o)
		\chi_0\big(-h(o, a.o; \omega_0)\big) {\: \rm d} n {\: \rm d} a
		\bigg| \\
		&\qquad\leqslant
		\bigg(\int_{N A} \big|h\big(a^{-1}n^{-1}.o\big)\big|^{p'} 
		\chi_0\big(-h(o, a.o; \omega_0)\big) {\: \rm d} n {\: \rm d} a \bigg)^{1/{p'}} \\
		&\qquad\qquad
		\bigg(
		\int_{N A} \big| f(nak g_0.o) \big|^p
		\chi_0\big(-h(o, a.o; \omega_0)\big) {\: \rm d} n {\: \rm d} a
		\bigg)^{1/p}.
	\end{align*}
	We observe that
	\begin{align*}
		\int_{N A} \big|h\big(a^{-1}n^{-1}.o\big)\big|^{p'} 
		\chi_0\big(-h(o, a.o; \omega_0)\big) {\: \rm d} n {\: \rm d} a
		=
		\int_K \int_{N A} \big|h\big(k^{-1} a^{-1}n^{-1}.o\big)\big|^{p'}
		{\: \rm d} n {\: \rm d} a {\: \rm d} k  = \|h\|_{\ell^{p'}}^{p'}.
	\end{align*}
	We next write
	\[
		k g_0 = n' a' k'.
	\]
	Hence, for $m = n a n' a^{-1} \in N$, we have
	\[
		n a k g_0 = n an'a^{-1} a a' k' = m a a' k'. 
	\]
	Consequently,
	\begin{align*}
		&
		\int_{N A} \big| f(nak g_0.o) \big|^p \chi_0\big(-h(o, a.o; \omega_0)\big) 
		{\: \rm d} n {\: \rm d} a \\
		&\qquad=
		\int_{N A} \big| f(m aa'.o) \big|^p \chi_0\big(-h(o, a.o; \omega_0)\big) 
		{\: \rm d} m {\: \rm d} a \\
		&\qquad=
		\int_{N A} \big| f(m b.o) \big|^p \chi_0\big(-h(a', b.o; \omega_0)\big)
		{\: \rm d} m {\: \rm d} b \\
		&\qquad=
		\chi_0\big(h(o, a'.o; \omega_0)\big) 
		\int_{N A} \big| f(m b.o) \big|^p \chi_0\big(-h(o, b.o; \omega_0)\big)
		{\: \rm d} m {\: \rm d} b\\
		&\qquad=
		\chi_0\big(h(o, k.x; \omega_0)\big) \|f\|_{\ell^p}^p.
	\end{align*}
	Since $K$ acts simply transitively on $\Omega$, and
	\[
		\chi_0\big(h(o, k.x; \omega_0)\big)^{\frac1p} = \chi_0\big(h(o, k.x; \omega_0)\big)^{\frac12}
		e^{\sprod{s_p}{h(o, k.x; \omega_0)}},
	\]
	we get
	\begin{align*}
		\int_K \chi_0\big(h(o, k.x; \omega_0)\big)^{\frac1p} {\: \rm d} k
		&=
		\int_\Omega 
		\chi_0\big(h(o, x; \omega_0)\big)^{\frac1p} \nu({\rm d} \omega) \\
		&=\vphi_{s_p}(x),
	\end{align*}
	which easily leads to \eqref{eq:74}. 

	Now, setting $\check{f}(\lambda) = f(\lambda^\star)$, by \eqref{eq:74}, we obtain
	\begin{align*}
		\sprod{f \times K}{g} 
		&=\sprod{K}{\check{f} \times g} \\
		&\leqslant \sum_{x \in V_g} |K(x)| \big| \check{f} \times g(x) \big| \\
		&\leqslant \sum_{x \in V_g} |K(x)| \vphi_{s_p}(x) \|\check{f}\|_{\ell^p} \|g\|_{\ell^{p'}} \\
		&\leqslant \calH(|K|)(s_p) \|f\|_{\ell^p} \|g\|_{\ell^{p'}},
	\end{align*}
	and the theorem follows.
\end{proof}

\begin{corollary}[Kunze--Stein phenomenon]
	\label{cor:A:2}
	Let $\mathbf{G}$ be a simply connected semisimple algebraic group defined over a locally compact non-Archimedean field,
	and let $\scrX$ be its Bruhat--Tits building. Let $p \in [1, 2)$. There is $C_p > 0$ such that for all radial
	$K \in \ell^p(V_g)$ and $f \in \ell^2(V_g)$,
	\[
		\|f \times K\|_{\ell^2} \leqslant C_p \|K\|_{\ell^p} \|f\|_{\ell^2}.
	\]
\end{corollary}
\begin{proof}
	Let us observe that $\vphi_0 \in \ell^{p'}(V_g)$. Indeed, by Proposition \ref{prop:1}, we have
	\begin{align*}
		\sum_{x \in V_g} 
		\vphi_0(x)^{p'} 
		&\approx
		\sum_{\lambda \in P^+} \Phi(\lambda)^{p'} \chi_0(\lambda) \\
		&\approx
		\sum_{\lambda \in P^+} \chi_0(\lambda)^{-\frac{p'}2+1} \prod_{\alpha \in \Phi^+} 
		\big(1 + |\sprod{\lambda}{\alpha}|\big) < \infty.
	\end{align*}
	Consequently, by \eqref{eq:76}, for a radial function $g \in \ell^2(V_g)$ we get
	\[
		\|f \times g\|_{\ell^{p'}} 	\leqslant \|f\|_{\ell^2} \|g\|_{\ell^2} \|\vphi_0\|_{\ell^{p'}}.
	\]
	Since $f \times g$ is radial, for each radial $K$ belonging to $\ell^p(V_g)$,
	\[
		\sprod{K}{f \times g} \leqslant \|K\|_{\ell^p} \|f\|_{\ell^2} \|g\|_{\ell^2} \|\vphi_0\|_{\ell^{p'}},
	\]
	thus
	\[
		\sprod{K \times \check{f}}{g} \leqslant \|K\|_{\ell^p} \|f\|_{\ell^2} \|g\|_{\ell^2} \|\vphi_0\|_{\ell^{p'}}.
	\]
	Hence,
	\[
		\|K \times \check{f}\|_{\ell^2} 
		\leqslant 
		\|K\|_{\ell^p} \|f\|_{\ell^2} 
		\|\vphi_0\|_{\ell^{p'}},
	\]
	which completes the proof.
\end{proof}

	\newcommand{\SL}{{\operatorname{SL}}} \newcommand{\SU}{{\operatorname{SU}}}
  \newcommand{\UU}{{\operatorname{U}}} \newcommand{\Sp}{{\operatorname{Sp}}}
  \newcommand{\SO}{{\operatorname{SO}}}

\end{document}